\documentclass[10pt,reqno]{amsart}

\usepackage{dsfont}
\usepackage[bottom]{footmisc}
\usepackage{amssymb, amsmath, amsthm}
\usepackage{enumerate,color}
\newtheorem{thm}{Theorem}[section]
\newtheorem{lem}[thm]{Lemma}

\newtheorem{prop}[thm]{Proposition}
\newtheorem{defn}[thm]{Definition}
\newtheorem{rmk}{Remark}

\numberwithin{equation}{section}

\newcommand{\bel}{\begin{equation} \label}
\newcommand{\ee}{\end{equation}}
\def\beq{\begin{equation}}
\def\eeq{\end{equation}}
\newcommand{\bea}{\begin{eqnarray}}
\newcommand{\eea}{\end{eqnarray}}
\newcommand{\beas}{\begin{eqnarray*}}
\newcommand{\eeas}{\end{eqnarray*}}
\newcommand{\pd}{\partial}

\newcommand{\wh}{\widehat}
\newcommand{\te}{\theta}

\newcommand{\re}{\mathfrak R}

\newcommand{\R}{\mathbb{R}}
\newcommand{\C}{\mathbb{C}} 
\newcommand{\N}{\mathbb{N}} 
\newcommand{\al}{\alpha}
\newcommand{\be}{\beta}
\newcommand{\Om}{\Omega}

\newcommand{\cA}{\mathcal{A}}
\newcommand{\cB}{\mathcal{B}}

\newcommand{\cS}{\mathcal{S}}

\allowdisplaybreaks

\def\epsilon{\varepsilon}
\def\phi {\varphi}

\providecommand{\abs}[1]{\left\lvert#1\right\rvert}
\providecommand{\norm}[1]{\left\lVert#1\right\rVert}

\renewcommand{\leq}{\leqslant}
\renewcommand{\geq}{\geqslant}

\providecommand{\abs}[1]{\left\lvert#1\right\rvert}
\providecommand{\norm}[1]{\left\lVert#1\right\rVert}

\title[Equivalence of definitions of solutions]
{\bf Equivalence of definitions of solutions for some class of fractional diffusion equations }

\author{Yavar Kian}
\address{Aix Marseille Univ, Université de Toulon, CNRS, CPT, Marseille, France}
\email{yavar.kian@univ-amu.fr}

\date{}

\begin{document}

\begin{abstract} We study the unique existence of weak solutions for  initial boundary value problems associated with different class of fractional diffusion equations including variable order, distributed order and multiterm fractional diffusion equations. So far, different definitions of weak solutions have been considered  for these class of problems. This includes definition of solutions in a variational sense and definition of solutions from properties of their Laplace transform in time. The goal of the present article is to unify these two approaches by showing the equivalence of these two definitions. Such property allows also to show that the weak solutions under consideration combine the advantage of these two class of solutions which include representation of solutions by a Duhamel type of formula, suitable properties of Laplace transform of solutions, resolution of the equation in the sense of distributions and explicit link with the initial condition. 

\end{abstract}

\maketitle


\section{Introduction}

\subsection{Settings}
\label{sec-settings}

Let $\Omega$ be a bounded and connected open subset of $\R^d$, $d \geq 2$, 
with Lipschitz  boundary $\partial \Omega$. 
Let $a:=(a_{i,j})_{1 \leq i,j \leq d} \in L^\infty(\Omega;\R^{d^2})\cap H^1(\Omega;\R^{d^2})$
be symmetric, that is 
$$ 
a_{i,j}(x)=a_{j,i}(x),\ x \in \Omega,\ i,j = 1,\ldots,d, 
$$
and fulfill the ellipticity condition: there exists a constant
$c>0$ such that
\bel{ell}
\sum_{i,j=1}^d a_{i,j}(x) \xi_i \xi_j \geq c |\xi|^2, \quad 
\mbox{for a.e. $x \in \Omega,\ \xi=(\xi_1,\ldots,\xi_d) \in \R^d$}.
\ee
Assume that $q \in L^{\frac{d}{2}}(\Omega)$ is non-negative and define the operator $\cA$ by
$$ 
\cA u(x) :=-\sum_{i,j=1}^d \partial_{x_i} 
\left( a_{i,j}(x) \partial_{x_j} u(x) \right)+q(x)u(x),\  x\in\Omega. 
$$ 
We set also $\rho \in L^\infty(\Omega)$ obeying
\bel{eq-rho}
 0<c_0 \leq\rho(x) \leq C_0 <+\infty,\ x \in \Omega. 
\ee

From now on we set $\R_+=(0,+\infty)$ and we introduce  the  function $K\in L^1_{loc}(\R_+;L^\infty(\Omega))\cap C^\infty(\R_+;L^\infty(\Omega))$ satisfying th following condition 
\bel{K} \inf\{\tau>0:\  e^{-\tau t}K\in L^1(\R_+;L^\infty(\Omega))\}=0.\ee
Then, we define the operator $I_K$ by
$$I_Kg(t,x)=\int_0^tK(t-s,x)g(s,x)ds,\quad g\in L^1_{loc}(\R_+;L^2(\Omega)),\ x\in\Omega,\ t\in\R_+.$$
We introduce the Caputo and Riemann-Liouville fractional derivative with kernel $K$ as follows
$$\partial^{K}_tg(t,x)=I_K\partial_tg(t,x),\quad D^{K}_tg(t,x)=\partial_tI_Kg(t,x),\quad g\in W^{1,1}_{loc}(\R_+;L^2(\Omega)),\ x\in\Omega,\ t\in\R_+.$$
In the present article we consider the following initial boundary value problem (IBVP):
\bel{eq1}
\left\{ \begin{array}{rcll} 
(\rho(x) \partial^{K}_t+\cA ) u(t,x) & = & F(t,x), & (t,x)\in 
\R_+\times\Omega,\\
 u(t,x) & = & 0, & (t,x) \in \R_+\times\partial\Omega, \\  
 u(0,x) & = & u_0(x), & x \in \Omega.
\end{array}
\right.
\ee
Namely, for different values of the kernel $K$, corresponding to variable order, distributed order and multiterm fractional diffusion equations, we prove the existence of weak solutions of \eqref{eq1} in the sense of a definition involving the Riemann-Liouville fractional derivative $D^{K}_t$. In addition, we would like to prove that this unique weak solution is described by a suitable Duhamel type of formula and its Laplace transform in time  has the expected  properties for such equations. Our goal is to unify  the two different main approaches considered so far for defining solutions of \eqref{eq1}. That is the definition of solutions of \eqref{eq1} in a variational sense involving Riemann-Liouville fractional derivative (see e.g. \cite{EK,Z}) and the definition of solutions in term of Laplace transform (see e.g. \cite{KSY,KY1,KY2,LKS}). 
\subsection{Definitions of weak and Laplace-weak solutions}
Before stating our results, we give the definition  of solutions under consideration in the present article. Inspired by \cite{EK,Z}, 
we give the definition of weak solutions of this initial boundary value problem (IBVP in short) \eqref{eq1} as follows.
\begin{defn}\label{d1} \emph{(Weak solution)}
Let the coefficients  in \eqref{eq1} satisfy
\eqref{ell}--\eqref{eq-rho}. We say that
$u\in L_{loc}^1(\R_+;L^2(\Omega))$ is a weak solution to \eqref{eq1} if it satisfies
the following conditions.
\begin{enumerate}
\item[{\rm(i)}] The following identity
\begin{equation}\label{d1a}
\rho( x)D_t^K[u-u_0](t, x)+\mathcal A u(t, x)=F(t,x),\quad  t\in\R_+,\ x\in\Omega
\end{equation}
holds true in the sense of distributions in $\R_+\times\Omega$.
\item[{\rm(ii)}] We have $I_K [u-u_0]\in W_{loc}^{1,1}(\R_+;D'(\Omega))$ and the following initial condition
\begin{equation}\label{d1b}
 I_K[u-u_0](0, x)=0,\quad  x\in\Omega,
\end{equation}
is fulfilled. 
\item[{\rm(iii)}] We have
$$
p_0=\inf\{\tau>0:\  e^{-\tau t}u\in L^1(\R_+;L^2(\Omega))\}<\infty
$$
and there exists $p_1\ge p_0$ such that for all $p\in\mathbb C$ satisfying $\re p>p_1$ we have
$$
\wh u(p,\,\cdot\,):=\int_0^\infty e^{-p t}u(t,\,\cdot\,)d t\in H^1_0(\Omega)
$$
\end{enumerate}
\end{defn}

\begin{rmk} The  conditions in Definition \ref{d1} describe the different aspects of the IBVP \eqref{eq1}. Namely, condition (i) is associated with the equation in \eqref{eq1}, condition (ii) describes the link with the initial condition of \eqref{eq1} and condition (iii) gives the boundary condition of \eqref{eq1}. Let us also observe that in the spirit of the works \cite{EK,Z} \emph{(}see also \cite{KRY,KuY,P}\emph{)}, we use in the equation \eqref{d1a} the fact that for $u\in W^{1,1}_{loc}(\R_+;L^2(\Omega))$ we have $\partial_t^K u=D_t^K[u-u(0,\cdot)]$. In that sense the expression $\partial_t^K u$ in \eqref{eq1} can be defined in a more general context by considering instead the expression $D_t^K[u-u(0,\cdot)]$.

\end{rmk}
In the present article we study the unique existence of a weak solution of the IBVP \eqref{eq1} in the sense of Definition \ref{d1} in three different context:\\
1) Variable order fractional diffusion equations where for $\alpha\in L^\infty(\Omega)$ satisfying
\bel{alpha}
0 < \alpha_0 \leq \alpha(x) \leq \alpha_M<1,\quad \alpha_M<2\alpha_0,\ x \in \Omega, 
\ee
we fix 
\bel{Kvariable} K(t,x)=\frac{t^{-\alpha(x)}}{\Gamma(1-\alpha(x))},\quad t\in\R_+,\ x\in\Omega.\ee
2) Distributed order fractional diffusion equations where for a non-negative weight function $\mu \in L^\infty(0,1)$, obeying the following condition:
\bel{mu}
\exists \alpha_0 \in(0,1),\ \exists \epsilon \in (0,\alpha_0),\ \forall \alpha \in (\alpha_0-\epsilon,\alpha_0),\ \mu(\alpha) \ge \frac{\mu(\alpha_0)}{2}>0,
\ee
we define
\bel{Kdistributed} K(t,x)=\int_0^1 \mu(\alpha)\frac{t^{-\alpha}}{\Gamma(1-\alpha)}d\alpha,\quad t\in\R_+,\ x\in\Omega.\ee
3) Multiple order fractional diffusion equations where, for $N\in\mathbb N$, $1<\alpha_1<\ldots<\alpha_N<1$ and for $\rho_j\in L^\infty(\Omega)$, $j=1,\ldots,N$, satisfying \eqref{eq-rho} with $\rho=\rho_j$, we fix
\bel{Kmultiple} K(t,x)=\sum_{j=1}^N\rho_j(x) \frac{t^{-\alpha_j}}{\Gamma(1-\alpha_j)},\quad t\in\R_+,\ x\in\Omega.\ee

Let us also recall an alternative definition of weak solutions of \eqref{eq1} defined in terms of Laplace transform (see e.g. \cite{KSY,KY1,KY2,LKS}). In order to distinguish these two definitions of solutions, in the remaining part of this article, this class of weak solutions will be called Laplace-weak solutions. From now on and in all the remaining parts of this article, we  denote by $\mathcal J$ the set of functions $F\in L^1_{loc}(\R_+;L^2(\Omega))$ for which  there exists $J\in\mathbb N$ such that
$t\mapsto(1+t)^{-J}F(t,\cdot)\in L^1(\R_+;L^2(\Omega))$.
Following \cite{KSY,KY1,KY2,LKS}, we give the following definition of Laplace-weak solutions of \eqref{eq1}.

\begin{defn}\label{d2}\emph{(Laplace-weak solution)} Assume that $K$ is given by either of the three expressions \eqref{Kvariable}, \eqref{Kdistributed} and \eqref{Kmultiple}. Let $F\in \mathcal J$ and let the coefficients and the source term in \eqref{eq1} satisfy
\eqref{ell}--\eqref{eq-rho}. We say that
$u\in L_{loc}^1(\R_+;L^2(\Omega))$ is a Laplace-weak solution to \eqref{eq1} if it satisfies
the following conditions.
\begin{enumerate}
\item[{\rm(i)}]
$\inf\{\tau>0:\  e^{-\tau t}u\in L^1(\R_+;L^2(\Omega))\}=0$.
\item[{\rm(ii)}] There exists $p_1\ge 0$ such that for all $p\in\mathbb C$ satisfying $\re p>p_1$, the Laplace transform $\wh u(p,\cdot)$ of
$u(t,\cdot\,)$ with respect to $t$  is lying in $H^1_0(\Omega)$ and it solves the following boundary value
problem
\begin{equation}\label{d2a}
\begin{cases}
(\mathcal A+\rho p\wh K(p,\cdot))\wh u(p,\cdot)=\wh F(p,\cdot)+\rho\wh K(p,\cdot)u_0 & \mbox{in }\Omega,\\
\wh u(p,\cdot\,)=0 & \mbox{on }\partial\Omega.
\end{cases}
\end{equation}
Note that here $\wh K(p,\cdot)$ is well defined for all $p\in\mathbb C$ satisfying $\re p>0$ thanks to condition \eqref{K}.
\end{enumerate}
\end{defn}
\begin{rmk} In Definition \ref{d2} all the  properties of the IBVP \eqref{eq1} are described by the boundary value problem \eqref{d2a}. 
Indeed, for a solution $u$ of \eqref{eq1} satisfying the condition
$$p_2=\inf\{\tau>0:\  e^{-\tau t}u\in W^{1,1}(\R_+;L^2(\Omega))\}<\infty,$$
we have 
$$\wh{\partial^{K}_t u}(p,\cdot\,)=\wh K(p,\cdot)[p\wh u(p,\cdot)-u(0,\cdot)]=p\wh K(p,\cdot))\wh u(p,\cdot)-\wh K(p,\cdot)u_0,\quad p\in\mathbb C,\ \re p>p_2.$$
Therefore, applying the Laplace transform in time to the equation of \eqref{eq1} we deduce that, for all $p\in\mathbb C$ satisfying $\re p>p_1$, $\wh u(p,\cdot)$ is the unique solution of \eqref{d2a}. Combining this with the uniqueness and the analyticity of the Laplace transform 
we can conclude that such solution of \eqref{eq1} coincides with the Laplace-weak solution of \eqref{eq1}. In that sense, this notion of Laplace-weak solutions allows to define solutions of \eqref{eq1} in terms of properties of their Laplace transform in time.
\end{rmk}
The goal of the present article is to unify these two definitions by proving the equivalence between Definition \ref{d1} and Definition \ref{d2} in some general context with the kernel $K$ given by \eqref{Kvariable}, \eqref{Kdistributed}, \eqref{Kmultiple}. For this purpose,   assuming that $F\in\mathcal J$ and $u_0\in L^2(\Omega)$, we will show the unique existence of Laplace-weak solutions of \eqref{eq1} in the sense of Definition \ref{d2}. After that we prove that the Laplace-weak solutions of \eqref{eq1} coincides with the unique weak solution of \eqref{eq1} in the sense of Definition \ref{d1}. This property shows in particular the  equivalence between these definitions. In addition to this equivalence, we give also a Duhamel type of representation of the weak solutions of \eqref{eq1} taking the form \eqref{sol1}, \eqref{di1} and \eqref{mul}. 

\subsection{Motivations}

Recall that anomalous diffusion in complex media have been intensively studied these last decades in different fields 
 with multiple applications in  geophysics, environmental 
 and biological problems. The diffusion properties of homogeneous media are 
currently modeled, see e.g. \cite{AG,CSLG}, by constant order 
time-fractional diffusion processes where in \eqref{eq1} the kernel $K$ takes the form $\frac{t^{-\beta}}{\Gamma(1-\beta)}$ with a constant values $\beta\in(0,1)$. However, in some complex media, several physical properties lead to more general model involving variable order, distributed order and multiterm fractional diffusion equations. For instance, it has been proved that the presence of heterogeneous regions displays space inhomogeneous variations and the constant order fractional dynamic models are not robust for long times (see \cite{FS}). In this context the variable order time-fractional model, corresponding to kernel $K$ given by \eqref{Kvariable},  is more relevant for describing the space-dependent anomalous diffusion process (see e.g. \cite{SCC}). In this context, several variable order diffusion models have been successfully applied in numerous applications in sciences and engineering, including Chemistry \cite{CZZ}, 
Rheology \cite{SdV}, Biology \cite{GN}, Hydrogeology \cite{AON} and Physics \cite{SS, ZLL}. In the same way, some anomalous diffusion process such as ultra-slow diffusion, where the mean squared variance grows only logarithmically with time, are modeled by  fractional diffusion equations with distributed order fractional derivatives with applications in polymer physics and kinetics of particles (see e.g. \cite{MMPG,MS}). For these different physical models, the goal of the present article is to prove existence of weak solutions of \eqref{eq1} enjoying several important properties such as resolution of the equation in the sense of distributions, suitable Duhamel representation formula and expected properties of the Laplace transform in time of the solutions stated in Definition \ref{d2}.

Beside these physical motivations, our analysis is also motivated by applications in other class of mathematical problems where the Duhamel representation formula, the properties of the Laplace transform in time of solutions as well as the resolution in the sense of distributions of the equation in \eqref{eq1}, stated in \eqref{d1a}-\eqref{d1b}, play an important role. This is for instance the case for several inverse problems (see e.g. \cite{JLLY,JK,KLLY,KLY,KOSY,KSXY,LIY}) as well as  the study of some dynamical properties  (see e.g. \cite{KSS,LLY}),  the derivation of analyticity properties in time of solutions (see e.g. \cite{LLY1}) and the numerical resolution (see e.g. \cite{B,JLSZ}) of these equations. In this context our goal is to exhibit weak solutions that satisfy simultaneously all the above mentioned properties.
 \subsection{Known results}

Recall that the well-posdness of the IBVP \eqref{eq1} has received a lot of attention these last decades among the mathematical community. For constant order fractional diffusion equations, where in \eqref{eq1} the kernel $K$ takes the form $\frac{t^{-\beta}}{\Gamma(1-\beta)}$ with a constant values $\beta\in(0,1)$, several approaches have been considered for defining solutions of \eqref{eq1}. This includes the definition of solutions of \eqref{eq1} in a variational and strong sense considered by \cite{EK,KuY,KRY,KY2,Z}, the definition  of solutions in the mild-sense in \cite{SY} and the definition of solutions by mean of their Laplace transform in time given by \cite{KY1,KY2}. Such analysis includes also the study of the IBVP \eqref{eq1} with a time-dependent elliptic operator as stated in \cite{KuY,KRY,Z}. Several  authors considered also the well-posedness of more general class of diffusion equations. For instance, the analysis of \cite{KJ,Z} in some abstract framework  can be applied to some class of distributed order and multiterm fractional diffusion equation of the form \eqref{eq1} with a kernel $K$ independent of $x$ (for $K$ given by \eqref{Kmultiple} the coefficients $\rho_1,\ldots,\rho_N$ are constants). In the same way, we can mention the work of \cite{KR,LKS} for the study of distributed order fractional diffusion equations and the work of \cite{LHY,LLY2} devoted to the study of well-posedness of multiterm fractional diffusion equations with both constant and variable coefficients $\rho_1,\ldots,\rho_N$ in \eqref{Kmultiple}. To the best of our knowledge, in the article \cite{KSY} one can find the only result available in the mathematical literature devoted to the study of the well-posedness of variable order fractional diffusion equations (the kernel $K$ given by \eqref{Kvariable}) with non-vanishing initial condition and general source term. In this last work, the authors give a definition of solutions in term of Laplace transform comparable to Definition \ref{d2}. As far as we know,  there is no result showing existence of weak solutions of \eqref{eq1} satisfying the properties described by Definition \ref{d1} for variable order fractional diffusion equations.

In all the above mentioned results the authors have either considered a variational definition of solutions comparable to Definition \ref{d1} 
or a definition of solutions in terms of Laplace transform comparable to the Laplace-weak solutions of Definition \ref{d2}. However, as far as we know, there has been no result so far proving the unification of these two definitions of solutions for variable order, distributed order or multiterm fractional diffusion equations.

\subsection{Main results}

The main results of this article state the unique existence of Laplace-weak solutions of \eqref{eq1} in the sense of Definition \ref{d2} as well as the equivalence between Definition \ref{d1} and \ref{d2}  for the weight function $K$ given by \eqref{Kvariable}, \eqref{Kdistributed}, \eqref{Kmultiple} which correspond to variable, distributed order and multiterm  fractional diffusion equations.

For variable order fractional diffusion equations, our result can be stated as follows.

\begin{thm}\label{t1}
Assume that the conditions \eqref{ell}-\eqref{eq-rho} are fulfilled. Let $u_0\in L^2(\Omega)$, $F\in\mathcal J$, $\alpha\in L^\infty(\Omega)$ satisfy \eqref{alpha} and let $K$ be given by \eqref{Kvariable}. Then there exists a unique Laplace-weak solution $u\in L^1_{loc}(\R_+;L^2(\Omega))$ of \eqref{eq1} in the sense of Definition \ref{d2}. Moreover, the Laplace-weak solution $u\in L^1_{loc}(\R_+;L^2(\Omega))$ of \eqref{eq1} is the unique weak solution of \eqref{eq1} in the sense of Definition \ref{d1}. In addition, $u$ is described by a Duhamel type of formula  taking the form \eqref{sol1}.
\end{thm}

For distributed order fractional diffusion equations, our result can be stated as follows.

\begin{thm}\label{t2}
Assume that the conditions \eqref{ell}-\eqref{eq-rho} are fulfilled. Let $u_0\in L^2(\Omega)$, $F\in\mathcal J$, $\mu \in L^\infty(0,1)$ be a non-negative function satisfying \eqref{mu} and let $K$ be given by \eqref{Kdistributed}. Then there exists a unique Laplace-weak solution $u\in L^1_{loc}(\R_+;L^2(\Omega))$ of \eqref{eq1} in the sense of Definition \ref{d2}. Moreover, the Laplace-weak solution $u\in L^1_{loc}(\R_+;L^2(\Omega))$ of \eqref{eq1} is the unique weak solution of \eqref{eq1} in the sense of Definition \ref{d1}. In addition, $u$ is described by a Duhamel type of formula  taking the form \eqref{di1}.
\end{thm}

For multiterm fractional diffusion equations, our result can be stated as follows.

\begin{thm}\label{t3}
Assume that the conditions \eqref{ell}-\eqref{eq-rho} are fulfilled with $\rho\equiv1$. Let $u_0\in L^2(\Omega)$, $F\in\mathcal J$, $1<\alpha_1<\ldots<\alpha_N<1$ and $\rho_j\in L^\infty(\Omega)$, $j=1,\ldots,N$, satisfying \eqref{eq-rho} with $\rho=\rho_j$, and let $K$ be given by \eqref{Kmultiple}. Then there exists a unique Laplace-weak solution $u\in L^1_{loc}(\R_+;L^2(\Omega))$ of \eqref{eq1} in the sense of Definition \ref{d2}. Moreover, the Laplace-weak solution $u\in L^1_{loc}(\R_+;L^2(\Omega))$ of \eqref{eq1} is the unique weak solution of \eqref{eq1} in the sense of Definition \ref{d1}. In addition, $u$ is described by a Duhamel type of formula  taking the form \eqref{mul}.
\end{thm}
\subsection{Comments about our results}

To the best of our knowledge, in Theorem \ref{t1} we obtain the first result of unique existence of weak solutions of variable order fractional fractional diffusion equations solving in the sense of distributions the equation in \eqref{eq1}, as stated in \eqref{d1a}, and with explicit connection to the initial condition $u_0$ stated in \eqref{d1a}-\eqref{d1b}. As far as we know, the only other comparable results can be found in \cite{KSY} where the authors proved only existence of Laplace-weak solution of \eqref{eq1} with $K$ given by \eqref{Kvariable}. In that sense, Theorem \ref{t1} gives the first extension of the the analysis of \cite{KSY} by proving that the unique Laplace-weak solution under consideration in \cite{KSY} is also the unique weak solution in the sense of Definition \ref{d1}.

Let us observe that, in Theorem \ref{t2} and \ref{t3} we show, for what seems to be the first time, the unique existence of weak solutions of distributed order and multiterm fractional diffusion equations that enjoy simultaneously the following properties; 1) The weak solution solves in the sense of distributions the equation in \eqref{eq1} as stated in \eqref{d1a}; 2) The weak solution is explicitly connected with the initial condition $u_0$ by \eqref{d1a}-\eqref{d1b}; 3) The weak solution is described by a Duhamel type of formula; 4) The weak solution is also a Laplace-weak solution in the sense of Definition \ref{d1}. Indeed, several authors proved unique existence of solutions of distributed order and multiterm fractional diffusion equations enjoying the above properties 1) and 2) (see e.g. \cite{KJ,KR,LHY}) or the above properties 3) and 4) (see e.g. \cite{KSY}). Nevertheless, we are not aware of any results proving existence of weak solutions of  distributed order  fractional diffusion equations  or multiterm fractional diffusion equations with variable coefficients enjoying simultaneously the above properties 1), 2), 3) and 4). In that sense,  Theorem \ref{t2} and \ref{t3} show  that these different properties  of solutions of  distributed order and multiterm fractional diffusion equations can be unified.

In contrast to Definition \ref{d2}, where the solutions are described by mean of the properties of Laplace transform in time of such class of fractional diffusion equations (see e.g. \cite{P} for more details), Definition \ref{d1} gives more explicit properties of solutions of \eqref{eq1}. Namely, the weak solution of \eqref{eq1}, in the sense of Definition \ref{d1}, solves the equation in \eqref{eq1} in the sense of distribution, as stated in \eqref{d1a}. Moreover, this class of weak solutions are also explicitly connected with the initial condition $u_0$ by mean of properties \eqref{d1a}-\eqref{d1b} and condition (iii) gives the boundary condition imposed to weak solutions in the sense of Definition \ref{d1}. By proving the equivalence between Definition \ref{d1} and Definition \ref{d2} of weak solutions, we show that the weak solution of \eqref{eq1} combine the explicit properties of Definition \ref{d1} with the properties of Laplace transform of solutions as stated in Definition \ref{d2}.

Let us observe that the results of Theorem \ref{t1}, \ref{t2} and \ref{t3} can be applied to the unique existence of solutions of the IBVP \eqref{eq1} at finite time (see the IBVP \eqref{eqq1}). This aspect is discussed in Section 5 of the present article with a definition of weak solutions stated in Definition \ref{d3} by mean of a weak  solution at infinite time in the sense of Definition \ref{d1}. Our results for these issue are stated in Theorem \ref{t4}, where we show that the unique solution in the sense of Definition \ref{d3} is independent of the choice of the final time.

Let us observe that the boundary condition  under consideration in \eqref{eq1} can be replaced, at the price of some minor modifications, by more general homogeneous Neumann or Robin boundary condition. In the spirit of the work \cite{KY2}, it is also possible to consider non-homogeneous boundary conditions. For simplicity we restrict our analysis to homogenous Dirichlet boundary conditions.

\subsection{Outline}

This paper is organized as follows. In Section 2, we prove the existence of a Laplace-weak solutions of the IBVP \eqref{eq1} in the sense of Definition \ref{d2}  as well as the equivalence between Definition \ref{d1} and  \ref{d2}, when $K$ is given by \eqref{Kvariable}, stated in Theorem \ref{t1}. In the same way, Section 3 and 4 are respectively devoted to the proof of Theorem \ref{t2} and \ref{t3}. Moreover, in Section 5, we study the same  problem at finite time (see the IBVP \eqref{eqq1})  and we give a definition of solutions in that context stated in Definition \ref{d3}. We prove also in Theorem \ref{t3} the unique existence of solutions in the sense of Definition \ref{d3} as well as the independence of the unique solution in the sense of  Definition \ref{d3} with respect to the final time.

\section{Variable order fractional diffusion equations}

In this section, we prove the unique existence of a weak
solution to the problem \eqref{eq1} as well as the equivalence between Definition \ref{d1} and \ref{d2} of weak and Laplace-weak solutions of \eqref{eq1} for weight $K$ given by \eqref{Kvariable} with $\alpha\in L^\infty(\Omega)$ satisfying \eqref{alpha}. For this purpose, let us first recall that the unique existence of solutions close to the  Laplace-weak solutions for \eqref{eq1} has been proved by \cite[Theorem 1.1]{KSY} in the case of source terms $F\in L^\infty(\R_+;L^2(\Omega))$. We will recall here the representation of Laplace-weak solutions of \eqref{eq1} given by \cite{KSY}.  For this purpose, we fix
$\theta\in(\frac{\pi}{2},\pi)$, $\delta>0$ and we define the contour in $\mathbb C$, 
\bel{cont1}
\gamma(\delta,\theta):=\gamma_-(\delta,\theta)\cup\gamma_0(\delta,\theta)\cup\gamma_+(\delta,\theta),\ee
oriented in the counterclockwise direction with
\begin{equation}\label{g2}
\gamma_0(\delta,\theta):=\{\delta\, e^{i\beta}:\ \beta\in[-\theta,\theta]\},\quad\gamma_\pm(\delta,\theta)
:=\{s\,e^{\pm i\theta}\mid s\in[\delta,\infty)\}.
\end{equation}
 We denote also by $A$ the Dirichlet realization of the operator $\mathcal A$ acting on $L^2(\Omega)$ with domain $H^2(\Omega)\cap H^1_0(\Omega)$. Then, following
\cite{KSY}, we define the operators
\begin{equation}\label{S0}
S_0(t)\psi:=\frac{1}{2i\pi}\int_{\gamma(\delta,\theta)}e^{t p}\left(A+\rho p^{\alpha(\cdot)}\right)^{-1}\rho p^{\alpha(\cdot)-1}\psi d p,
\quad t>0,
\end{equation}
\begin{equation}\label{S1}
S_1(t)\psi:=\frac{1}{2i\pi}\int_{\gamma(\delta,\theta)}e^{t p}\left(A+\rho p^{\alpha(\cdot)}\right)^{-1}\psi d p,
\quad t>0.
\end{equation}
According to
\cite[Theorem 1.1]{KSY}, the definition of the operator valued functions $S_0,S_1$ are independent of the choice of
$\theta\in\left(\frac{\pi}{2},\pi\right)$, $\delta>0$. In light of \cite[Theorem 1.1]{KSY} and
\cite[Remark 1]{KSY}, for $u_0\in L^2(\Omega)$ and $F\in L^\infty(\R_+;L^2(\Omega))$, the function $u$ defined by 
\begin{equation}\label{sol1}
u(t, \cdot )=S_0(t)u_0+
\int_0^t S_1(t-\tau)F(\tau,\cdot) d\tau,\quad t>0,
\end{equation}
is the unique tempered distribution with respect to the time variable $t\in\R_+$ taking values in $L^2(\Omega)$ whose Laplace transform in time solves \eqref{d2a}. This means that the function $u$ given by \eqref{sol1} will be the unique Laplace-weak solution of problem \eqref{eq1} in the sense of Definition \ref{d1} provided that  $u\in L^1_{loc}(\R_+;L^2(\Omega))$.

We start by proving an extension of this result to the  unique existence of a Laplace-weak solution of problem \eqref{eq1} when $u_0\in L^2(\Omega)$ and $F\in\mathcal J$.  For this purpose, we need the following intermediate result about the operator valued functions $S_0$ and $S_1$.

\begin{lem}\label{l1}
Let $\theta\in\left(\frac{\pi}{2},\pi\right)$. The maps $t\longmapsto S_j(t)$, $j=0,1$, defined by
\eqref{S0}-\eqref{S1} are lying in $L^1_{loc}(\R_+;\mathcal B(L^2(\Omega))$ and there
exists a constant $C>0$ depending only on $\mathcal A,\rho,\alpha,\theta,\Omega$
such that the estimates
\begin{equation}\label{l1a}
\|S_0(t)\|_{\mathcal B(L^2(\Omega))}\leq
C\max\left(t^{2(\alpha_M-\alpha_0)},
t^{2(\alpha_0-\alpha_M)},1\right),\quad t>0,
\end{equation}
\begin{equation}\label{ll2a}
\|S_1(t)\|_{\mathcal B(L^2(\Omega))}\leq
C\max\left(t^{2\alpha_M-\alpha_0-1},t^{2\alpha_0-\alpha_M-1},1\right),\quad t>0,
\end{equation}
hold true
\end{lem}
\begin{proof} 
Throughout this proof, by $C>0$ we denote generic constants
depending only on $\cA,\rho,\alpha,\te,\Om$, which may change from line to line. In this lemma we only consider the proof of this lemma for the operator valued function $S_0$, for $S_1$ one can refer to \cite[Lemma 6.1]{KLY} for the proof of \eqref{ll2a}.
In light of \cite[Proposition 2.1]{KSY}, for all $\beta\in(0,\pi)$, we have
\begin{equation}\label{ll2b}
\left\|\left(A+\rho(r\,e^{i\beta_1})^{\alpha(\cdot)}\right)^{-1}\right\|_{\mathcal B(L^2(\Omega ))}
\le C\max\left(r^{\alpha_0-2\alpha_M},r^{\alpha_M-2\alpha_0}\right),
\quad r>0,\ \be_1\in(-\be,\be).
\end{equation}
Using the fact that the operator $S_0$ is independent of the choice of $\delta>0$, we can
decompose
\[
S_0(t)=H_-(t)+H_0(t)+H_+(t),\quad t>0,
\]
where
\[
H_m(t)=\frac{1}{2i\pi}\int_{\gamma_m(t^{-1},\te)} e^{t p}(\rho p^{\alpha(\cdot)}+A)^{-1}\rho p^{\alpha-1} dp,
\quad m=0,\mp,\ t>0.
\]
In order to complete the proof of the lemma, it suffices to prove
\begin{equation}\label{ll2c}
\|H_m(t)\|_{\mathcal B(L^2(\Omega))}\leq C\max\left(t^{2(\alpha_M-\alpha_0)},t^{2(\alpha_0-\alpha_M)},1\right),\quad t>0,\ m=0,\mp.
\end{equation}
Indeed, these estimates clearly implies \eqref{l1a}. Moreover,  condition \eqref{alpha} implies that $2(\alpha_0-\alpha_M)>-\alpha_M>-1$ and we deduce from \eqref{l1a} that $S_0\in L^1_{loc}(\R_+;\mathcal B(L^2(\Omega))$. For $m=0$, using \eqref{ll2b}, we find
\[\begin{aligned}
\|H_0(t)\|_{\mathcal B(L^2(\Omega))}&\leq C\int_{-\theta}^\theta
t^{-1}\left\|\left(A+(t^{-1}e^{i\be})^{\alpha(\cdot)}\right)^{-1}\right\|_{\mathcal B(L^2(\Omega)}\norm{|t^{-1}e^{i\be}|^{\alpha(\cdot)-1}}_{L^\infty(\Omega)}
d\beta\\
&\leq C\max\left(t^{2(\alpha_M-\alpha_0)},t^{2(\alpha_0-\alpha_M)},1\right),\end{aligned}
\]
which implies \eqref{ll2c} for $m=0$. For $m=\mp$, again we
employ \eqref{ll2b} to estimate
$$\begin{aligned}
\|H_\mp(t)\|_{\mathcal B(L^2(\Omega))} & \leq C\int_{t^{-1}}^\infty e^{r t\cos\theta}
\left\|\left(A+(r\,e^{i\te})^{\al(\cdot)}\right)^{-1}\right\|_{\mathcal B(L^2(\Om)}d r\\
& \leq C\int_{t^{-1}}^\infty e^{r t\cos\theta}
\max\left(r^{\alpha_0-2\alpha_M},r^{\alpha_M-2\alpha_0}\right) \norm{r^{\alpha(\cdot)-1}}_{L^\infty(\Omega)}d r\\
& \leq C\int_{t^{-1}}^\infty e^{r t\cos\theta}
\max\left(r^{2(\alpha_0-\alpha_M)-1},r^{2(\alpha_M-\alpha_0)-1}\right) d r
\end{aligned}$$
For $t>1$, we obtain
$$\begin{aligned}
\|H_\mp(t)\|_{\mathcal B(L^2(\Omega))} & \leq C\int_1^\infty e^{r t\cos\theta}
r^{2(\alpha_M-\alpha_0)-1}d r+C\int_{t^{-1}}^1r^{2(\alpha_0-\alpha_M)-1} d r\\
& \leq C\int_0^\infty e^{r t\cos\te}r^{2(\alpha_M-\alpha_0)-1}\,d r
+C\left(t^{2(\alpha_0-\alpha_M)}+1\right)\\
& \leq C\,t^{-1}\int_0^\infty e^{r\cos\te}
\left(\frac{ r}{ t}\right)^{2(\alpha_M-\alpha_0)-1}d r
+C\left(t^{2(\alpha_0-\alpha_M)}+1\right)\\
& \leq C\max\left(t^{2(\alpha_0-\alpha_M)},t^{2(\alpha_M-\alpha_0)},1\right).
\end{aligned}$$
In the same way, for $t\in(0,1]$, we get
\[
\|H_\mp(t)\|_{\mathcal B(L^2(\Om))}\leq C\int_1^\infty e^{r t\cos\te}
r^{2(\alpha_M-\alpha_0)-1}d r\leq C\,t^{2(\alpha_M-\alpha_0)}.
\]
Combining these two estimates, we obtain
\[
\|H_\mp(t)\|_{\mathcal B(L^2(\Omega))}\leq
C\max\left(t^{2(\alpha_0-\alpha_M)},t^{2(\alpha_M-\alpha_0)},1\right),\quad t>0.
\]
This proves that \eqref{ll2c} also holds true for $m=\mp$. Therefore, estimate \eqref{l1a} holds true and we have $S_0\in L^1_{loc}(\R_+;\mathcal B(L^2(\Omega))$ which completes  the proof of the lemma.
\end{proof}

We are now in position to state the existence of a unique Laplace-weak solution $u\in L^1_{loc}(\R_+;L^2(\Omega))$ of \eqref{eq1}, given by \eqref{sol1}, for any source term $F\in\mathcal J$.

\begin{prop}\label{p1} Assume that the conditions \eqref{ell}-\eqref{eq-rho} are fulfilled.
Let $u_0\in L^2(\Omega)$, $F\in\mathcal J$, $\alpha\in L^\infty(\Omega)$ satisfy \eqref{alpha} and let $K$ be given by \eqref{Kvariable}. Then there exists a unique Laplace-weak
solution $u\in L^1_{loc}(\R_+;L^2(\Omega))$ to \eqref{eq1} given by \eqref{sol1}.
\end{prop}

\begin{proof}

According to \cite[Theorem 1.1]{KSY} and Lemma \ref{l1}, we only need to prove this result for $u_0\equiv0$.  Using Lemma \ref{l1},  we will complete the proof of Proposition \ref{p1} by mean of density arguments.  Fix $$G(t,x)=(1+t)^{-J}F(t,x),\quad (t,x)\in\R_+\times\Omega$$
and recall that $G\in L^1(\R_+;L^2(\Omega))$. Therefore, we can find a sequence $(G_n)_{n\in\mathbb N}$ lying in $\mathcal C^\infty_0(\R_+\times\Omega)$ such that
$$\lim_{n\to\infty}\|G_n-G\|_{L^1(\R_+;L^2(\Omega))}=0.$$
Fixing $(F_n)_{n\in\mathbb N}$ a sequence of functions defined by 
$$F_n(t,x)=(1+t)^{J}G_n(t,x),\quad (t,x)\in\R_+\times\Omega,\ n\in\mathbb N,$$
we deduce that the sequence $(F_n)_{n\in\mathbb N}$ is lying in $\mathcal C^\infty_0(\R_+\times\Omega)$ and we have
\begin{equation}\label{p1a}
\lim_{n\to\infty}\|(1+t)^{-J}(F_n-F)\|_{L^1(\R_+;L^2(\Omega))}=\lim_{n\to\infty}\|G_n-G\|_{L^1(\R_+;L^2(\Omega))}=0.
\end{equation}
In light of Lemma \ref{l1}, for $t> 0$ we can introduce
$$\begin{aligned}
u_n(t,\cdot\,) & :=\int_0^t S_1(t-\tau)F_n(\tau,\cdot)d\tau=\int_0^t S_1(\tau)F_n(t-\tau,\cdot)d\tau,
\quad n\in\mathbb N,\\
u(t,\,\cdot\,) & :=\int_0^t S_1(t-\tau)F(\tau,\cdot)d\tau
\end{aligned}$$
as elements of $L^1_{loc}(\R_+;L^2(\Omega))$. We will prove that for all $p\in\mathbb C_+:=\{z\in\mathbb C:\ \re z>0\}$, the
Laplace transform $\wh u(p)$ of $u$ is well-defined in $L^2(\Omega)$ and we
have
\begin{equation}\label{p1c}
\lim_{n\to\infty}\|\wh{u_n}(p)-\wh u(p)\|_{L^2(\Omega)}=0.
\end{equation}
Applying estimate \eqref{ll2a}, we obtain
$$\begin{aligned}
\left\|e^{-p t}u(t,\,\cdot\,)\right\|_{L^2(\Omega)} & \leq\int_0^t
e^{-\re p\tau}\|S_1(\tau)\|_{\mathcal B(L^2(\Omega))}\,
e^{-\re p(t-\tau)}\|F(t-\tau,\cdot)\|_{L^2(\Omega)}d\tau\\
& \leq C\left(e^{-\re p t}\max
\left(t^{2\alpha_0-\al_M-1},t^{2\alpha_M-\al_0-1},1\right)\right)*
\left(e^{-\re p t}\|F(t,\cdot)\|_{L^2(\Omega)}\right)
\end{aligned}$$
for all $t>0$ and $p\in\mathbb C_+$, where $*$ denotes the convolution in
$\mathbb R_+$. Therefore, applying Young's convolution inequality
and condition \eqref{alpha}, we deduce
$$\begin{aligned}
&\|\wh u(p)\|_{L^2(\Omega)}\\
 & \leq\int_0^\infty
\left\| e^{-p t}u(t,\,\cdot\,)\right\|_{L^2(\Omega)}d t\\
& \leq C\left(\int_0^\infty e^{-\re p t}
\max\left(t^{2\al_0-\al_M-1},t^{2\al_M-\al_0-1},1\right)d t\right)
\left(\int_0^\infty e^{-\re p t}\|F(t,\cdot)\|_{L^2(\Omega)}d t\right)\\
&\leq C_p\left(\int_0^\infty e^{-\re p t}
\max\left(t^{2\al_0-\al_M-1},t^{2\al_M-\al_0-1},1\right)d t\right) \|(1+t)^{-J}F(t,\cdot)\|_{L^1(\R_+;L^2(\Omega))}<\infty.
\end{aligned}$$
for all $p\in\mathbb C_+$. This proves that $\wh u(p)$ is well defined for
all $p\in\mathbb C_+$ in the sense of $L^2(\Omega)$. In the same way, for all
$t>0$, $p\in\mathbb C_+$ and $n\in\mathbb N$, we get
$$\begin{aligned}
&\left\| e^{-p t}(u_n-u)(t,\,\cdot\,)\right\|_{L^2(\Omega)}\\
& \leq\int_0^te^{-\re  p(t-\tau)}\|S_1(t-\tau)\|_{\mathcal B(L^2(\Omega))}
e^{-\re  p\tau}\|F_n(\tau,\cdot)-F(\tau,\cdot)\|_{L^2(\Omega)}d\tau\\
& \le C\left(e^{-\re p t}
\max\left(t^{2\al_0-\al_M-1},t^{2\al_M-\al_0-1},1\right)\right)*
\left(e^{-\re p t}\|F_n(t,\cdot)-F(t,\cdot)\|_{L^2(\Omega)}\right).
\end{aligned}$$
Thus, applying Young's convolution inequality again, we have
$$\begin{aligned}
& \quad\,\|\wh{u_n}(p)-\wh u(p)\|_{L^2(\Omega)}
\le\int_0^\infty\left\|e^{-p t}(u_n-u)(t,\,\cdot\,)\right\|_{L^2(\Omega)}d t\\
& \le C\left(\int_0^\infty e^{-\re  p t}
\max\left(t^{2\al_0-\al_M-1},t^{2\al_M-\al_0-1},1\right)d t\right)
\left(\int_0^\infty e^{-\re  p t}\|F_n(t,\cdot)-F(t,\cdot)\|_{L^2(\Omega)}\,d t\right)\\
& \le C_p\left(\sup_{t\in\R_+}(1+t)^{J}e^{-\re  p t}\right)\|(1+t)^{-J}(F_n-F)\|_{L^1(\R_+;L^2(\Omega))}
\end{aligned}$$
for all $p\in\mathbb C_+$ and $n\in\mathbb N$, and \eqref{p1a} implies \eqref{p1c}. On the other
hand, in view of \cite[Theorem 1.1 and Remark 1]{KSY}, since
$F_n\in L^\infty(\R_+;L^2(\Omega))$, we have
\[
\left(\rho p^{\alpha(\cdot)}+A\right)\wh{u_n}(p,\cdot)=\left(\int_0^\infty e^{-p t}F_n(t,\,\cdot\,)\,d t\right)
,\quad p\in\mathbb C_+,\ n\in\N.
\]
 In addition, \eqref{p1a} implies that, for all $p\in\mathbb C_+$, we have
\[\begin{aligned}
\limsup_{n\to\infty}\norm{\wh{F_n}(p)-\wh{F}(p)}_{L^2(\Omega)}&\leq \limsup_{n\to\infty}\int_0^\infty e^{-\re  p t}\|F_n(t,\cdot)-F(t,\cdot)\|_{L^2(\Omega)}\,d t\\
\ &\leq  \left(\sup_{t\in\R_+}(1+t)^{N}e^{-\re  p t}\right) \limsup_{n\to\infty}\|(1+t)^{-J}(F_n-F)\|_{L^1(\R_+;L^2(\Omega))}\\
\ &\leq0.\end{aligned}
\]
Therefore, we obtain
\[
\lim_{n\to\infty}
\left\|\wh{u_n}(p)-\left(A+\rho p^{\alpha(\cdot)}\right)^{-1}\wh F(p,\cdot)\right\|_{L^2(\Omega)}=0
\]
and \eqref{p1c} implies that $\wh u(p)=\left(A+\rho p^{\alpha(\cdot)}\right)^{-1}\wh F(p,\cdot)$,
$p\in\mathbb C_+$. From the definition of the operator $A$, we deduce that
$\wh u(p)\in H^1_0(\Omega)$ solves the boundary value problem \eqref{d2a} for
all $p\in\mathbb C_+$. Recalling that the uniqueness of Laplace-weak solutions can be deduced easily from the uniqueness of the solution of \eqref{d2a} and the uniqueness of Laplace transform, we conclude that $u$ is the unique Laplace-weak
solution of \eqref{eq1} and the proof is completed.
\end{proof}

In view of Proposition \ref{p1} the first statement of Theorem \ref{t1} is fulfilled. Let us complete the proof of Theorem \ref{t1}.

\textbf{Proof of Theorem $\ref{t1}$.} Let us first observe that the first statement of Theorem \ref{t1} is a direct consequence of Proposition \ref{p1}. Therefore, in order to complete the proof of Theorem \ref{t1} we need to prove that \eqref{eq1} admits a unique weak solution $u\in L^1_{loc}(\R_+;L^2(\Omega))$ in the sense of Definition \ref{d1} given by \eqref{sol1} which is the unique Laplace-weak solution of \eqref{eq1}. We divide the proof of our result into three steps. We start by proving the uniqueness of the solution of \eqref{eq1} in the sense of Definition \ref{d2}. Then we prove that \eqref{sol1} is a weak solution of \eqref{eq1} in the sense of Definition \ref{d1} for $u_0\equiv0$. Finally, we show that \eqref{sol1}  is a weak solution of \eqref{eq1} in the sense of Definition \ref{d1} for $F\equiv0$.

{\bf Step 1. } This step will be devoted to the proof of the uniqueness of weak solutions of \eqref{eq1} in the sense of Definition \ref{d1}.
For this purpose, let $u\in L^1_{loc}(\R_+;L^2(\Omega))$ be a weak solution of \eqref{eq1} with $F\equiv0$ and $u_0\equiv0$. In view of condition (iii) of Definition \ref{d1}, we can fix
$$
p_0=\inf\{\tau>0:\ e^{-\tau t}u\in L^1(\R_+;L^2(\Omega))\}.
$$
Then, using the fact that $a:=(a_{i,j})_{1 \leq i,j \leq d} \in  H^1(\Omega;\R^{d^2})$, for all $p\in\mathbb C$ satisfying $\re\,p>p_0$, we have $e^{-p t}\mathcal A u\in L^1(\R_+;D'(\Omega))$ and condition \eqref{d2a} implies that 
$$
e^{-p t}\rho D_t^K u=-e^{-p t}\mathcal A u,\quad t\in\R_+.
$$
It follows that, for all $p\in\mathbb C$ satisfying $\re\,p>p_0$, $e^{-p t}\rho D_t^K u\in L^1(\R_+;D'(\Omega))$. Combining this with condition (ii) of Definition \ref{d1}, we deduce that, for all $p\in\mathbb C$ satisfying $\re\,p>p_0$, $e^{-p t}\rho I_Ku\in W^{1,1}(\R_+;D'(\Omega))$. Therefore, multiplying \eqref{d1a} by $e^{-p t}$ with $p\in\mathbb C$ satisfying $\re\,p>p_0$, integrating over $t\in\R_+$ and using condition \eqref{d1b}, we find
\begin{align*}
0 & =\int_0^\infty e^{-p t}\left(\rho D_t^K u(t,\,\cdot\,)+\mathcal A u(t,\,\cdot\,)\right)d t=\int_0^\infty e^{-p t}\left(\partial_t[\rho I_Ku(t,\,\cdot\,)]+\mathcal A u(t,\,\cdot\,)\right)d t\\
& =p\rho\wh{I_Ku}(p,\,\cdot\,)+\wh{\mathcal A u}(p,\,\cdot\,)=\rho p^{\alpha(\cdot)}\wh u(p,\,\cdot\,)+\mathcal A\wh u(p,\,\cdot\,).
\end{align*}
Combining this with condition (iii) of Definition \ref{d2}, we deduce that, for all $p\in\mathbb C$ satisfying $\re\,p>p_1$, $\wh u(p,\,\cdot\,)\in H^1_0(\Omega)$ solves the boundary value problem
$$
\begin{cases}
(\mathcal A+p^{\alpha(\cdot)}\rho)\wh u(p,\,\cdot\,)=0 & \mbox{in }\Omega,\\
\wh u(p,\,\cdot\,)=0 & \mbox{on }\partial\Omega.
\end{cases}
$$

On the other hand, applying \cite[Proposition 2.1]{KSY} we deduce that for all $p\in\mathbb C_+$ the operator $A+p^{\alpha(\cdot)}\rho$ is invertible as an operator acting on $L^2(\Omega)$. Therefore, we get that , for all $p\in\mathbb C$ satisfying $\re\,p>p_1$,
$\wh u(p;\,\cdot\,)\equiv0$ and, combining this with the analyticity and the uniqueness of Laplace transform, we deduce that $u\equiv0$. This completes the proof of the uniqueness of weak solution of problem \eqref{eq1}.\medskip

{\bf Step 2.} In this step we will prove that the Laplace-weak solution $u\in L^1_{loc}(\R_+;L^2(\Omega))$ of \eqref{eq1}, given by \eqref{sol1}, is the weak solution of \eqref{eq1} in the sense of Definition \ref{d1} when $u_0\equiv0$. Note that the Laplace-weak solution $u\in L^1_{loc}(\R_+;L^2(\Omega))$ of \eqref{eq1} clearly satisfies condition (iii) of Definition \ref{d1}. Thus, we only need to prove that the Laplace-weak solution $u\in L^1_{loc}(\R_+;L^2(\Omega))$ of \eqref{eq1} satisfies conditions (i) and (ii) of Definition \ref{d1}. In a similar way to Proposition \ref{p1}, we fix  a sequence $(G_n)_{n\in\mathbb N}$ lying in $\mathcal C^\infty_0(\R_+\times\Omega)$ such that
$$\lim_{n\to\infty}\|G_n-G\|_{L^1(\R_+;L^2(\Omega))}=0$$
and $(F_n)_{n\in\mathbb N}$ a sequence of functions of $\mathcal C^\infty_0(\R_+\times\Omega)$ defined by 
$$F_n(t,x)=(1+t)^{J}G_n(t,x),\quad (t,x)\in\R_+\times\Omega,\ n\in\mathbb N.$$
Then condition \eqref{p1a} is fulfilled. According to Proposition \ref{p1}, for all $n\in\mathbb N$, the Laplace-weak solution $u_n$ of \eqref{eq1} with $F=F_n$ is given by
$$
u_n(t,\,\cdot\,)=\int_0^t\,S(t-\tau)F_n(\tau,\cdot)d\tau=\int_0^t\,S(\tau)F_n(t-\tau,\cdot)d\tau,\quad t\in \mathbb R_+.
$$
Using the fact that $F_n\in\mathcal C^\infty_0(\R_+\times\Omega)$, $n\in\mathbb N$, and applying estimate \eqref{ll2a}, we deduce that $u_n\in C^1([0,\infty);L^2(\Omega))$ and $u_n(0,x)=0,\  x\in\Omega$.
Moreover, in view of \eqref{ll2a}, applying Young's convolution inequality, for all $p\in\mathbb C_+$, we get
$$\begin{aligned}
& \quad\,\| e^{-p t}u_n\|_{L^1(\R_+;L^2(\Omega))}+\|e^{-p t}\partial_tu_n\|_{L^1(\R_+;L^2(\Omega))}\\
& \leq C\left\|\left(e^{-(\re\,p)t}(\|F_n(t)\|_{L^2(\Omega)}+\|\partial_tF_n(t)\|_{L^2(\Omega)})\right)*\left(e^{-(\re\,p)t}\|S(t)\|_{\mathcal B(L^2(\Omega))}\right)\right\|_{L^1(\R_+)}\\
& \leq C\|F_n\|_{W^{1,1}(\R_+;L^2(\Omega))}\int_0^\infty e^{-(\re\,p)t}\max\left(t^{2\alpha_M-\alpha_0-1},t^{2\alpha_0-\alpha_M-1},1\right)d t<\infty.
\end{aligned}$$
Thus, for all $n\in\mathbb N$, we have $D_t^K u_n=\partial_t^K u_n$ and we deduce that, for all $p\in\mathbb C_+$, we have
$$
\wh{D_t^K u_n}(p,\,\cdot\,)=p^{\alpha(\cdot)}\wh{u_n}(p,\,\cdot\,).
$$
Therefore, using the fact that for all $p\in\mathbb C$ satisfying $\re\,p>p_1$, $\wh{u_n}(p,\,\cdot\,)$ solves \eqref{d2a} with $F=F_n$ and $u_0\equiv0$, we deduce that
$$
\wh{D_t^K u_n}(p,\,\cdot\,)=p^{\alpha(\cdot)}\wh{u_n}(p,\,\cdot\,)=-\rho^{-1}\mathcal A\wh u_n(p,\,\cdot\,)+\rho^{-1}\wh{F_n}(p,\cdot)=\wh{w_n}(p,\,\cdot\,),\quad p\in\mathbb C,\ \re\,p>p_1,
$$
where $w_n(t, x)=-\mathcal A u_n(t, x)+F_n(t,x)$, $(t, x)\in \R_+\times\Omega$.
Combining this with the uniqueness and the analyticity of the Laplace transform in time of $u_n$, we deduce that the identity
\[
\rho(x) D_t^{K}u_n(t, x)+\mathcal A u_k(t, x)=F_n(t,x),\quad (t, x)\in\R_+\times\Omega
\]
holds true in the sense of distributions on $\R_+\times\Omega$. In the same way, using the fact that $u_n\in C^1([0,\infty);L^2(\Omega))$ with $u_n(0,\,\cdot\,)\equiv0$, we deduce that $I_Ku_n\in C^1([0,\infty);L^2(\Omega))$ and
\begin{equation}\label{t1c}
I_Ku_n(0, x)=0,\quad  x\in\Omega.
\end{equation}
From now on, we will prove that the above properties can be extended by density to $u$. Fix $T_1>0$. Applying this last identity, we will show that $D_t^K u_n$ converges in the sense of $L^1(0,T_1;D'(\Omega))$ to $-\mathcal A u(t, x)+F(t,x)$ as $n\to\infty$, and then we will complete the proof of the theorem.

Define the space $L^2(\Omega;\rho dx)$ corresponding to the space of function $L^2$ on $\Omega$ with a measure of density $\rho$. We define the operator $A_*=\rho^{-1}\mathcal A$ acting on $L^2(\Omega;\rho dx)$ with domain $D(A_*)=\{v\in H^1_0(\Omega):\ \rho^{-1}\mathcal A v\in L^2(\Omega)\}$ and we recall that $A_*$ is a selfadjoint operator with a compact resolvent whose spectrum consists of a non-deacreasing unbounded positive eigenvalues. In view of Proposition \ref{p1}, we have $u\in L^1_{loc}(\R_+;L^2(\Omega))$ and we deduce that $\rho^{-1}\mathcal A u\in L^1(0,T_1;D(A_*^{-1}))$. Moreover, 
there exists $C>0$ depending only on $\mathcal A$, $\rho$ and $\Omega$ such that, for all $n\in\mathbb N$, we have
\begin{equation}\label{t1d}
\|\rho^{-1}\mathcal A (u_n- u)\|_{L^1(0,T_1;D(A_*^{-1}))}\leq C\|u_n-u\|_{L^1(0,T_1;L^2(\Omega))}.
\end{equation}
In addition, we have
$$
u_n(t,\,\cdot\,)-u(t,\,\cdot\,)=\int_0^tS(t-\tau)[F_n(\tau)-F(\tau)]d\tau,\quad t\in\R_+,\ n\in\mathbb N,
$$ and applying Lemma \ref{l1} and Young's convolution inequality, we obtain
$$\begin{aligned}
\|u_n-u\|_{L^1(0,T_1;L^2(\Omega))}&\leq C\|t^{2\alpha_0-\alpha_M-1}\|_{L^1(0,T_1)}\|F_n-F\|_{L^1(0,T_1;L^2(\Omega))}\\
&\leq C\|(1+t)^{-J}F_n-F\|_{L^1(0,T_1;L^2(\Omega))}\leq C\|(1+t)^{-J}F_n-F\|_{L^1(\R_+;L^2(\Omega))}\end{aligned}
$$
and \eqref{p1a} implies that
$$
\lim_{n\to\infty}\|\rho^{-1}\mathcal A u_n-\rho^{-1}\mathcal A u\|_{L^1(0,T_1;D(A_*^{-1}))}=0.
$$
In the same way, we have 
$$
\lim_{n\to\infty}\|\rho^{-1}F_n-\rho^{-1}F\|_{L^1(0,T_1;D(A_*^{-1}))}\leq C\lim_{n\to\infty}\|(1+t)^{-J}F_n-F\|_{L^1(\R_+;L^2(\Omega))} =0
$$
and it follows that $(D_t^K u_n)_{n\in\mathbb N}$ converges in the sense of $L^1(0,T_1;D(A_*^{-1}))$ to $-\rho^{-1}\mathcal A u+F$ as $n\to\infty$. On the other hand, for all $\psi\in  C^\infty_0(0,T_1)$, we have 
\bel{t1e}\begin{aligned}
\langle D_t^K u_n(t,\,\cdot\,),\psi(t)\rangle_{D'(0,T_1),C^\infty_0(0,T_1)} & =\left\langle\partial_t I_Ku_k(t,\,\cdot\,),\psi(t) \right\rangle_{D'(0,T_1),C^\infty_0(0,T_1)}\nonumber\\
& =-\left\langle I_Ku_n(t,\,\cdot\,),\psi'(t)\right\rangle_{D'(0,T_1),C^\infty_0(0,T_1)}.
\end{aligned}\ee
In addition, repeating the above arguments and applying \eqref{p1a}, one can check that the sequence $(u_n)_{n\in\mathbb N}$ converges to $u$ in the sense of $L^1(0,T_1;L^2(\Omega))$ as $n\to\infty$ and, applying again Young's convolution inequality, we deduce that the sequence $I_Ku_n$ converges to $I_Ku$
in the sense of $L^1(0,T_1;L^2(\Omega))$ as $n\to\infty$. Thus, sending $n\to\infty$, we find
$$\begin{aligned}
\lim_{n\to\infty}\langle D_t^K u_n(t,\,\cdot\,),\psi(t)\rangle_{D'(0,T_1),C^\infty_0(0,T_1)} & =-\lim_{n\to\infty}\left\langle I_Ku_n(t,\,\cdot\,),\psi'(t)\right\rangle_{D'(0,T_1),C^\infty_0(0,T_1)}\\
& =-\left\langle I_Ku(t,\,\cdot\,),\psi'(t) \right\rangle_{D'(0,T_1),C^\infty_0(0,T_1)}\\
& =\left\langle \partial_tI_Ku(t,\,\cdot\,),\psi(t)\right\rangle_{D'(0,T_1),C^\infty_0(0,T_1)}\\
& =\langle D_t^K u(t,\,\cdot\,),\psi(t)\rangle_{D'(0,T_1),C^\infty_0(0,T_1)}.
\end{aligned}$$
It follows that $D_t^K u_n$ converges in the sense of $D'(0,T_1;L^2(\Omega))$ to $D_t^K u$ as $n\to\infty$. Therefore, by the uniqueness of the limit in the sense of $D'(0,T_1;D(A_*^{-1}))$, we deduce that 
\[
D_t^K u(t,x)=-\rho^{-1}\mathcal A u(t, x)+\rho^{-1}(x)F(t,x),\quad t\in(0,T_1),\ x\in\Omega
\]
holds true and $D_t^K u\in L^1(0,T_1;D(A_*^{-1}))$. Using the fact that $T_1>0$ is arbitrarily chosen, we deduce that condition (i) of Definition \ref{d1} holds true. In addition, using the fact $D_t^K u=\partial_t I_Ku$, we obtain that $I_Ku\in W^{1,1}_{loc}(\R_+;D(A_*^{-1}))$. Recalling that $\mathcal C^\infty_0(\Omega)\subset D(A_*)$, one can check that $D(A_*^{-1})\subset D'(\Omega)$ which implies that $I_Ku\in W^{1,1}_{loc}(\R_+;D'(\Omega))$. Combining this with \eqref{p1a}, \eqref{t1d} and fixing $T_1=1$, we deduce that
$$\begin{aligned}
& \quad\,\limsup_{n\to\infty}\|I_Ku_n-I_Ku\|_{W^{1,1}(0,1;D(A_*^{-1}))}\\
& \leq\limsup_{n\to\infty}\|I_Ku_n-I_Ku\|_{L^1(0,1;D(A_*^{-1}))}+\|D_t^K u_k-D_t^K u\|_{L^1(0,1;D(A_*^{-1}))}\\
& \leq C\left(\limsup_{n\to\infty}\|u_n-u\|_{L^1(0,1;L^2(\Omega))}+\limsup_{k\to\infty}\|-\rho^{-1}[\mathcal A (u_k-u)+F_n-F] \|_{L^1(0,1;D(A_*^{-1}))}\right)\\
& \leq0.
\end{aligned}$$
Thus, we have that $(I_Ku_n)_{n\in\mathbb N}$ converge to $I_Ku$ in the sense of $W^{1,1}(0,1;D(A_*^{-1}))$ and from 
\eqref{t1c} we deduce that \eqref{d1b} is fulfilled. This proves that the Laplace-weak solution $u\in L^1_{loc}(\R_+;L^2(\Omega))$ of \eqref{eq1}, given by \eqref{sol1}, satisfies condition (ii) of Definition \ref{d1} which implies that $u$ satisfies all the conditions of Definition \ref{d1}. Thus \eqref{sol1} is a weak solution of \eqref{eq1} in the sense of Definition \ref{d1}. This completes the proof of the theorem when $u_0\equiv0$.

{\bf Step 3.} In this step we will prove that the Laplace-weak solution $u\in L^1_{loc}(\R_+;L^2(\Omega))$ of \eqref{eq1}, given by \eqref{sol1}, is the weak solution of \eqref{eq1} in the sense of Definition \ref{d1} when $F\equiv0$. Again, since the Laplace-weak solution $u\in L^1_{loc}(\R_+;L^2(\Omega))$ of \eqref{eq1} clearly satisfies condition (iii) of Definition \ref{d1}, we only need to check conditions (i) and (ii). For this purpose, let us first consider the following intermediate result.
\begin{lem}\label{l3} Let $\delta>0$ and $\theta\in(\pi/2,\pi)$. Then we have
\bel{l3a}\frac{1}{2i\pi}\int_{\gamma(\delta,\theta)}\frac{e^{t p}}{p} d p=1,\quad t\in\R_+.\ee
\end{lem}
We postpone the proof this lemma to the end of the present demonstration. Fix $(u_{0,n})_{n\in\mathbb N}$ a sequence of $\mathcal C^\infty_0(\Omega)$ such that 
\bel{t1f} \lim_{n\to\infty}\norm{u_{0,n}-u_0}_{L^2(\Omega)}=0\ee
and consider
$$u_n(t,\cdot)=S_0(t)u_{0,n}.$$
Fix $n\in\mathbb N$. In view of \eqref{l3a} and \eqref{S0}, we have
$$\begin{aligned}u_n(t,\cdot)-u_{0,n}&=\frac{1}{2i\pi}\int_{\gamma(\delta,\theta)}\frac{e^{t p}}{p} \left(A+\rho p^{\alpha(\cdot)}\right)^{-1}\rho p^{\alpha(\cdot)} u_{0,n}-\left(\frac{1}{2i\pi}\int_{\gamma(\delta,\theta)}\frac{e^{t p}}{p} d p\right)u_{0,n}\\
\ &=\frac{1}{2i\pi}\int_{\gamma(\delta,\theta)}\frac{e^{t p}}{p} \left[\left(A+\rho p^{\alpha(\cdot)}\right)^{-1}\rho p^{\alpha(\cdot)} u_{0,n}-u_{0,n}\right]dp.\end{aligned}$$
On the other hand, for all $p\in\mathbb C\setminus(-\infty,0]$, we have
$$\left(A+\rho p^{\alpha(\cdot)}\right)^{-1}\rho p^{\alpha(\cdot)} u_{0,n}-u_{0,n}=-\left(A+\rho p^{\alpha(\cdot)}\right)^{-1}Au_{0,n}$$
and it follows that
$$u_n(t,\cdot)-u_{0,n}=-\frac{1}{2i\pi}\int_{\gamma(\delta,\theta)}\frac{e^{t p}}{p}\left(A+\rho p^{\alpha(\cdot)}\right)^{-1}Au_{0,n}dp.$$
On the other hand, in view of \cite[Proposition 2.1]{KSY} there exists a constant $C>0$ depending on $\mathcal A$, $\rho$, $\Omega$ and $\alpha$ such that, for all $p\in \gamma(\delta,\theta)$, we have
\bel{t1g}
\left\|\left(A+\rho p^{\alpha(\cdot)}\right)^{-1}\right\|_{\mathcal B(L^2(\Omega ))}
\leq C\max\left(|p|^{\alpha_0-2\alpha_M},|p|^{\alpha_M-2\alpha_0}\right).\ee
Therefore, for all $p\in \gamma(\delta,\theta)$ and all $t\in\R_+$, we have
$$\norm{\frac{e^{t p}}{p}\left(A+\rho p^{\alpha(\cdot)}\right)^{-1}Au_{0,n}}_{L^2(\Omega)}\leq e^{t\re p}C\max\left(|p|^{\alpha_0-2\alpha_M-1},|p|^{\alpha_M-2\alpha_0-1}\right),$$
$$\norm{\partial_t\frac{e^{t p}}{p}\left(A+\rho p^{\alpha(\cdot)}\right)^{-1}Au_{0,n}}_{L^2(\Omega)}\leq e^{t\re p}C\max\left(|p|^{\alpha_0-2\alpha_M},|p|^{\alpha_M-2\alpha_0}\right).$$
It follows that 
\bel{t1i}\partial_tu_n(t,\cdot)=\partial_t[u_n(t,\cdot)-u_{0,n}]=-\frac{1}{2i\pi}\int_{\gamma(\delta,\theta)}e^{t p}\left(A+\rho p^{\alpha(\cdot)}\right)^{-1}Au_{0,n}dp=-S_1(t)Au_{0,n}.\ee
Combining this with Lemma \ref{l1} and the fact that $Au_{0,n}\in L^2(\Omega)$, we deduce that the map $u_n$ is lying in $W^{1,1}_{loc}(\R_+;L^2(\Omega))$. Moreover, we have 
$$u_n(0,\cdot)-u_{0,n}=-\frac{1}{2i\pi}\int_{\gamma(\delta,\theta)}p^{-1}\left(A+\rho p^{\alpha(\cdot)}\right)^{-1}Au_{0,n}dp.$$
Let us prove that
\bel{t1h}u_n(0,\cdot)-u_{0,n}=-\frac{1}{2i\pi}\int_{\gamma(\delta,\theta)}p^{-1}\left(A+\rho p^{\alpha(\cdot)}\right)^{-1}Au_{0,n}dp\equiv0.\ee
For this purpose, we fix $\delta<1$, $R>1$ and we consider the contour 
\bel{cont2}\gamma(\delta,R,\theta):=\gamma_-(\delta,R,\theta)\cup\gamma_0(\delta,\theta)\cup\gamma_+(\delta,R,\theta)\ee
oriented in the counterclockwise direction, where $\gamma_0(\delta,\theta)$ is given by \eqref{g2} and where
\[\gamma_\pm(\delta,R,\theta):=\{r e^{\pm i\theta}:\  r\in(\delta,R)\}.
\]
Applying the Cauchy formula,  for any  $R>1$, we have
$$\frac{1}{2i\pi}\int_{\gamma(\delta,R,\theta)}p^{-1}\left(A+\rho p^{\alpha(\cdot)}\right)^{-1}Au_{0,n}d p=\frac{1}{2i\pi}\int_{\gamma_0(R,\theta)}p^{-1}\left(A+\rho p^{\alpha(\cdot)}\right)^{-1}Au_{0,n} d p$$
with $\gamma_0(R,\theta)$ given by \eqref{g2} with $\delta=R$. Sending $R\to+\infty$, we obtain 
$$\frac{1}{2i\pi}\int_{\gamma(\delta,\theta)}p^{-1}\left(A+\rho p^{\alpha(\cdot)}\right)^{-1}Au_{0,n}dp=\lim_{R\to+\infty}\frac{1}{2i\pi}\int_{\gamma_0(R,\theta)}p^{-1}\left(A+\rho p^{\alpha(\cdot)}\right)^{-1}Au_{0,n}dp.$$
On the other hand, applying \eqref{t1g}, we deduce that
$$\begin{aligned}&\norm{\frac{1}{2i\pi}\int_{\gamma_0(R,\theta)}p^{-1}\left(A+\rho p^{\alpha(\cdot)}\right)^{-1}Au_{0,n} d p}_{L^2(\Omega)}\\
&\leq C\int_{-\theta}^{\theta}\norm{\left(A+\rho(Re^{i\beta})^{\alpha(\cdot)}\right)^{-1}}_{\mathcal B(L^2(\Omega)}\norm{Au_{0,n}}_{L^2(\Omega)} d \beta\\
\ &\leq C\max\left(R^{\alpha_0-2\alpha_M},R^{\alpha_M-2\alpha_0}\right)\norm{Au_{0,n}}_{L^2(\Omega)}.\end{aligned}$$
In view of \eqref{alpha}, it follows
$$\frac{1}{2i\pi}\int_{\gamma(\delta,\theta)}p^{-1}\left(A+\rho p^{\alpha(\cdot)}\right)^{-1}Au_{0,n}d p=\lim_{R\to+\infty}\frac{1}{2i\pi}\int_{\gamma_0(R,\theta)}p^{-1}\left(A+\rho p^{\alpha(\cdot)}\right)^{-1}Au_{0,n} d p\equiv0.$$
This proves \eqref{t1f} and  in a similar way to Step 2, we deduce that  $I_K[u_n-u_{0,n}]\in W^{1,1}_{loc}(\R_+;L^2(\Omega))$ satisfies
\begin{equation}\label{t1h}
I_K[u_n-u_{0,n}](0,x)=0,\quad  x\in\Omega.
\end{equation}
Therefore, $u_n$ satisfies condition (ii) of Definition \ref{d1}. Now let us show that $u_n$ satisfies condition (ii) of Definition \ref{d1}. Applying Lemma \ref{l1}  and \eqref{t1i}, we deduce that there exists a constant $C>0$ such that, for all $t>0$, we have
$$\begin{aligned}&\norm{u_n(t,\cdot)}_{L^2(\Omega)}+\norm{\partial_tu_n(t,\cdot)}_{L^2(\Omega)}\\
&\leq \norm{S_0(t)u_{0,n}}_{L^2(\Omega)}+\norm{S_1(t)A u_{0,n}}_{L^2(\Omega)}\\
\ &\leq C\max\left(t^{2(\alpha_M-\alpha_0)},
t^{2(\alpha_0-\alpha_M)},t^{2\alpha_M-\alpha_0-1},t^{2\alpha_0-\alpha_M-1},1\right)\norm{u_{0,n}}_{D(A)}.\end{aligned}$$
Combining this with \eqref{alpha}, for all $p\in\mathbb C_+$, we obtain $t\mapsto e^{-pt}I_K[u_n(t,\cdot)-u_{0,n}]\in W^{1,1}(\R_+;L^2(\Omega))$. Thus, for all $n\in\mathbb N$ and all $p\in\mathbb C_+$, fixing $v_n(t,\cdot)=u_n(t,\cdot)-u_{0,n}$, $t>0$, and applying \eqref{t1h}, we get
$$\begin{aligned}
\wh{D_t^K v_n}(p,\,\cdot\,)&=\int_0^{+\infty}e^{-pt}\partial_tI_K [u_n(t,\cdot)-u_{0,n}]dt\\
&=p\int_0^{+\infty}e^{-pt}I_K [u_n(t,\cdot)-u_{0,n}]dt\\
&=p\left(\int_0^{+\infty}e^{-pt}I_K u_n(t,\cdot)dt-u_{0,n}\int_0^{+\infty}e^{-pt}\frac{t^{1-\alpha(\cdot)}}{\Gamma(2-\alpha(\cdot))}dt\right)\\
&=p^{\alpha(\cdot)}\wh{u_n}(p,\,\cdot\,)- p^{\alpha(\cdot)-1}u_{0,n}.\end{aligned}
$$
Therefore, using the fact that for all $p\in\mathbb C$ satisfying $\re\,p>p_1$, $\wh{u_n}(p,\,\cdot\,)$ solves \eqref{d2a} with $F\equiv0$ and $u_0=u_{0,n}$, we deduce that
$$
\wh{D_t^K v_n}(p,\,\cdot\,)=p^{\alpha(\cdot)}\wh{u_n}(p,\,\cdot\,)- p^{\alpha(\cdot)-1}u_{0,n}=-\rho^{-1}\mathcal A\wh u_n(p,\,\cdot\,),\quad p\in\mathbb C,\ \re\,p>p_1.
$$
Then in a similar way to Step 2, we find that the identity \[
\rho(x) D_t^{K}[u_n-u_{0,n}](t, x)+\mathcal A u_n(t, x)=0,\quad (t, x)\in\R_+\times\Omega
\]
holds true in the sense of distributions on $\R_+\times\Omega$. We will now extend this result by density to the Laplace-weak solution of \eqref{eq1} which is given by \eqref{sol1} with $F\equiv0$. For this purpose, fix $T_1>0$.  Let us first observe that applying Lemma \ref{l1}, we obtain
$$\lim_{n\to+\infty}\|u_n-u\|_{L^1(0,T_1;L^2(\Omega))}\leq C\|t^{2(\alpha_0-\alpha_M)}\|_{L^1(0,T_1)}\lim_{n\to+\infty}\|u_{0,n}-u_0\|_{L^2(\Omega)}=0.$$
Therefore, repeating the arguments of Step 2, we can prove  that $D_t^K u_n$ converges in the sense of $D'(0,T_1;L^2(\Omega))$ to $D_t^K u$ and in the sense of $L^1(0,T_1;D'(\Omega))$ to $\rho^{-1}\mathcal A u$ as $n\to\infty$. Then, repeating the arguments used in the last part of Step 2 we deduce that the Laplace-weak solution $u$ of \eqref{eq1} fulfills the condition (i) and (ii) of Definition \ref{d1}. This completes the proof of Theorem \ref{t1}.\qed

Now that we have completed the proof of Theorem \ref{t1}, let us consider the proof of Lemma \ref{l3}.

\textbf{Proof of Lemma $\ref{l3}$.} Let us first recall that for all $t>0$ the map $z\mapsto \frac{e^{t z}}{z}$ is meromorphic on $\mathbb C$ with a simple pole at  $z=0$. Therefore, the residue theorem implies that for all $R>\delta$ we have 
\bel{l3b}\frac{1}{2i\pi}\int_{\gamma(\delta,R,\theta)}\frac{e^{t p}}{p} d p=1+\frac{1}{2i\pi}\int_{\gamma_1(R,\theta)}\frac{e^{t p}}{p} d p,\quad t\in\R_+,\ee
where we recall that $\gamma$ is given by \eqref{cont2} and $\gamma_1(R,\theta)$ is given by 
$$\gamma_1(R,\theta):=\{R\, e^{i\beta}:\ \beta\in[\theta, 2\pi-\theta]\}.$$
 Sending $R\to+\infty$, we obtain 
$$\frac{1}{2i\pi}\int_{\gamma(\delta,\theta)}\frac{e^{t p}}{p}dp=\lim_{R\to+\infty}\frac{1}{2i\pi}\int_{\gamma(\delta,R,\theta)}\frac{e^{t p}}{p} d p.$$
On the other hand, we have
$$\begin{aligned}\abs{\frac{1}{2i\pi}\int_{\gamma_1(R,\theta)}\frac{e^{t p}}{p} d p}&\leq \frac{1}{2\pi}\int_{\theta}^{2\pi-\theta} e^{tR\cos \beta}  d \beta\\
\ &\leq C e^{tR\cos\theta}.\end{aligned}$$
In view of \eqref{l3b} and the fact that $\theta\in(\pi/2,\pi)$, we find
$$\frac{1}{2i\pi}\int_{\gamma(\delta,\theta)}\frac{e^{t p}}{p}d p=1+\lim_{R\to+\infty}\frac{1}{2i\pi}\int_{\gamma_1(R,\theta)}\frac{e^{t p}}{p} d p=1,\quad t\in\R_+.$$ This proves \eqref{l3a} and it completes the proof of Lemma \ref{l3}.\qed

\section{Distributed order fractional diffusion equations}

In this section, we prove the unique existence of a weak
solution to the problem \eqref{eq1} as well as the equivalence between Definition \ref{d1} and \ref{d2} of weak and Laplace-weak solutions of \eqref{eq1} for weight $K$ given by \eqref{Kdistributed} with $\mu\in L^\infty(0,1)$ a non-negative function satisfying \eqref{mu}. For this purpose, let us first recall that the unique existence of Laplace-weak solutions for \eqref{eq1} has been proved by \cite{LKS} in the case of source terms $F\in L^\infty(\R_+;L^2(\Omega))$ and extended to source terms $F\in L^1(\R_+;L^2(\Omega))$ by \cite[Proposition 5.1]{KSXY}. We will recall here the representation of Laplace-weak solutions of \eqref{eq1} given by these works. Like in the previous section we denote by $A_*$ the operator $\rho^{-1}\mathcal A$ acting in the space $L^2(\Omega;\rho dx)$ with Dirichlet boundary condition. Let $(\varphi_n)_{n\geq1}$ be an $L^2(\Omega;\rho d x)$ orthonormal basis of eigenfunctions of the operator $A_*$ associated with the non-decreasing sequence of eigenvalues $(\lambda_n)_{n\geq1}$
of $A_*$ repeated with respect to there multiplicity. According to  \cite[Proposition 2.1]{LKS}, the unique Laplace weak solution $u$ of \eqref{eq1} enjoys the following representation formula
\bel{di1}
u(t,\cdot)=S_{0,\mu}(t)u_0+ \int_0^t S_{1,\mu}(t-s) F(s,\cdot) ds,\ t \in \R_+,
\ee
where
\bel{Smu0} 
S_{0,\mu}(t) \psi := \sum_{n=1}^\infty \frac{1}{2i\pi}  \int_{\gamma(\delta,\theta)}  e^{pt}(\lambda_n+\vartheta(p))^{-1} \frac{\vartheta(p)}{p}\left\langle \psi,\phi_n\right\rangle_{L^2(\Omega;\rho dx)}\phi_n,\ \psi \in L^2(\Omega),
\ee
\bel{Smu1} 
S_{1,\mu}(t) \psi := \sum_{n=1}^\infty \frac{1}{2i\pi}  \int_{\gamma(\delta,\theta)}  e^{pt}(\lambda_n+\vartheta(p))^{-1} \left\langle \rho^{-1}\psi,\phi_n\right\rangle_{L^2(\Omega;\rho dx)}\phi_n,\ \psi \in L^2(\Omega),
\ee
where $\vartheta(p):=\int_0^1 p^{\alpha} \mu(\alpha) d \alpha$, $\theta\in(\pi/2,\pi)$, $\delta>0$ and $\gamma(\delta,\theta)$ corresponds to the contour \eqref{cont1}. According to \cite{LKS}, the map $S_{j,\mu}$ is independent of the choice of $\theta\in(\pi/2,\pi)$, $\delta>0$. We start by proving an extension of this result  to source terms $F\in\mathcal J$ and by proving that the Laplace-weak solution $u$ given by 
\eqref{di1} is lying in $L^1_{loc}(\R_+;L^2(\Omega))$. For this purpose, like in the previous section, we need the following intermediate result about the operator valued functions $S_{0,\mu}$ and $S_{1,\mu}$.

\begin{lem}\label{l4}
Let $\theta\in\left(\frac{\pi}{2},\pi\right)$. The maps $t\longmapsto S_{\mu,j}(t)$, $j=0,1$, defined by
\eqref{Smu0}-\eqref{Smu1} are lying in $L^1_{loc}(\R_+;\mathcal B(L^2(\Omega))$ and there
exists a constant $C>0$ depending only on $\mathcal A,\rho,\theta,\mu,\Omega$
such that the estimates
\begin{equation}\label{l4a}
\|S_{0,\mu}(t)\|_{\mathcal B(L^2(\Omega))}\leq
C\max\left(t^{\alpha_0-\epsilon-1},t^{\alpha_0}\right),\quad t>0,
\end{equation}
\begin{equation}\label{l4b}
\|S_{1,\mu}(t)\|_{\mathcal B(L^2(\Omega))}\leq
C\max\left(t^{\alpha_0-\epsilon-1},t^{\alpha_0-1}\right),\quad t>0,
\end{equation}
hold true.
\end{lem}
\begin{proof} For the proof of these results for $S_{1,\mu}$ one can refer to \cite[Proposition 5.1]{KSXY} and we only show the above properties for $S_{0,\mu}$. For this purpose, we recall the following estimate from \cite[Lemma 2.2]{LKS},
\bel{est3}
\frac{1}{\abs{\vartheta(p)+\lambda_n}}\leq C\max(\abs{p}^{-\alpha_0+\epsilon},\abs{p}^{-\alpha_0}),\ p \in \C \setminus (-\infty,0],\ n \in \N,
\ee
where the positive constant $C$ depends only on $\mu$. We recall also that
$$
|\vartheta(p)|\leq C\max(\abs{p},1),\ p \in \C \setminus (-\infty,0],
$$
where the positive constant $C$ depends only on $\mu$. Therefore, we have
\bel{l1c}
\frac{|\vartheta(p)|}{|p|\abs{\vartheta(p)+\lambda_n}}\leq C\max(\abs{p}^{-\alpha_0+\epsilon},\abs{p}^{-1-\alpha_0}),\ p \in \C \setminus (-\infty,0],\ n \in \N,
\ee
where the positive constant $C$ depends only on $\mu$.
For all $t \in (0,+\infty)$ and all $\psi \in L^2(\Omega)$,  by taking $\delta=t^{-1}$ in \eqref{cont1},  \eqref{l1c} implies
\bel{l1d}\begin{aligned}
&  \norm{\sum_{n=1}^{+\infty}  \left( \int_{\gamma_0(t^{-1},\theta)} \frac{\vartheta(p) e^{pt}}{p(\vartheta(p)+\lambda_n)} dp \right) \langle \psi, \varphi_{n}\rangle_{L_\rho^2(\Omega)} \varphi_{n,k}}_{L^2(\Omega)}\\
&\leq  C\norm{\sum_{n=1}^{+\infty}  \left( \int_{\gamma_0(t^{-1},\theta)} \frac{\vartheta(p)e^{pt}}{p(\vartheta(p)+\lambda_n)} dp \right) \langle \psi, \varphi_{n}\rangle_{L_\rho^2(\Omega;\rho dx)} \varphi_{n,k}}_{L^2(\Omega;\rho dx)}\\
&\leq  C \max(t^{\alpha_0-\epsilon-1},t^{\alpha_0}) \left(\int_{-\theta}^{\theta}e^{\cos \beta}d\beta\right) \norm{\psi}_{L^2(\Omega;\rho dx)}\leq C  \max(t^{\alpha_0-\epsilon-1},t^{\alpha_0})\norm{\psi}_{L^2(\Omega)},\end{aligned}
\ee

\beas
& & \norm{\sum_{n=1}^{+\infty}
\left( \int_{\gamma_\pm(t^{-1},\theta)}\frac{\vartheta(p) e^{pt}}{p(\vartheta(p)+\lambda_n)} dp \right)
\langle \psi,\varphi_{n}\rangle_{L^2(\Omega;\rho dx)} \varphi_{n}}_{L^2(\Omega)}\\
&\leq & C \left(\int_{t^{-1}}^{+\infty} \max( r^{-\alpha_0+\epsilon},r^{-\alpha_0-1} ) e^{t r \cos \theta} dr\right) \norm{\psi}_{L^2(\Omega)}\\
&\leq & C t^{-1} \left(\int_{1}^{+\infty} \max ( (t^{-1} r)^{-\alpha_0+\epsilon}, (t^{-1} r)^{-\alpha_0-1}  ) e ^{r \cos \theta} dr\right) \norm{\psi}_{L^2(\Omega)}\\
&\leq&  C \max(t^{\alpha_0-\epsilon-1},t^{\alpha_0}) \left(\int_{1}^{+\infty} e^{r\cos\theta}dr\right) \norm{\psi}_{L^2(\Omega)}.
 \eeas
Putting these two estimates together with \eqref{g2} we deduce \eqref{l4a} and the fact that $S_{0,\mu}\in L^1_{loc}(\R_+;\mathcal B(L^2(\Omega))$.

\end{proof}

Combining Lemma \ref{l4} with the arguments used in Proposition \ref{p1}, we obtain the following result about the unique existence of Laplace-weak solution for problem \eqref{eq1}.

\begin{prop}\label{p2} Assume that the conditions \eqref{ell}-\eqref{eq-rho} are fulfilled.
Let $u_0\in L^2(\Omega)$, $F\in\mathcal J$, $\mu\in L^\infty(0,1)$ be a non-negative function satisfying \eqref{mu} and let $K$ be given by \eqref{Kdistributed}. Then there exists a unique Laplace-weak
solution $u\in L^1_{loc}(\R_+;L^2(\Omega))$ to \eqref{eq1} given by \eqref{di1}.
\end{prop}

Armed with this result, we can now complete the proof of Theorem \ref{t2}.

\textbf{Proof of Theorem $\ref{t2}$.} Let us first observe that the first statement of Theorem \ref{t2} is a direct consequence of Proposition \ref{p2}. Moreover, the uniqueness of weak solutions in the sense of Definition \ref{d1} as well as the fact that, for $u_0\equiv0$,  \eqref{eq1} admits a unique weak solution $u\in L^1_{loc}(\R_+;L^2(\Omega))$ in the sense of Definition \ref{d1} given by \eqref{di1}, can be deduced by mimicking the proof of Theorem \ref{t1}. For this purpose, we only show that $u$ given by \eqref{di1} is a weak solution of \eqref{eq1} in the sense of Definition \ref{d1} when $F\equiv0$. Since the Laplace-weak solution $u\in L^1_{loc}(\R_+;L^2(\Omega))$ of \eqref{eq1} clearly satisfies condition (iii) of Definition \ref{d1}, we only need to check condition (i) and (ii).  Fix $(u_{0,n})_{n\in\mathbb N}$ a sequence of $\mathcal C^\infty_0(\Omega)$ such that \eqref{t1f} is fulfilled
and consider
$$u_n(t,\cdot)=S_{0,\mu}(t)u_{0,n},\quad t\in\R_+.$$
Fix $n\in\mathbb N$. In view of \eqref{l3a} and \eqref{Smu0}, we have
$$\begin{aligned}&u_n(t,\cdot)-u_{0,n}\\
&=\sum_{k=1}^\infty \left(\frac{1}{2i\pi}  \int_{\gamma(\delta,\theta)}  \frac{\vartheta(p)e^{pt}}{p(\lambda_k+\vartheta(p))} dp\right)\left\langle u_{0,n},\phi_k\right\rangle_{L^2(\Omega;\rho dx)}\phi_k-\left(\frac{1}{2i\pi}\int_{\gamma(\delta,\theta)}\frac{e^{t p}}{p} d p\right)u_{0,n}\\
&=\sum_{k=1}^\infty \left(\frac{1}{2i\pi}  \int_{\gamma(\delta,\theta)} \frac{e^{pt}}{p} \left( \frac{\vartheta(p)}{(\lambda_k+\vartheta(p))} -1\right)dp\right)\left\langle u_{0,n},\phi_k\right\rangle_{L^2(\Omega;\rho dx)}\phi_k\\
&=-\sum_{k=1}^\infty \left(\frac{1}{2i\pi}  \int_{\gamma(\delta,\theta)} \frac{e^{pt}}{p}  \frac{\lambda_k}{(\lambda_k+\vartheta(p))} dp\right)\left\langle u_{0,n},\phi_k\right\rangle_{L^2(\Omega;\rho dx)}\phi_k\\
&=-\sum_{k=1}^\infty \left(\frac{1}{2i\pi}  \int_{\gamma(\delta,\theta)}   \frac{e^{pt}}{p(\lambda_k+\vartheta(p))}dp\right) \left\langle A_*u_{0,n},\phi_k\right\rangle_{L^2(\Omega;\rho dx)}\phi_k.\end{aligned}$$
On the other hand, applying \eqref{est3}, for all $k\in\mathbb N$ and all $p\in\mathbb C\setminus(-\infty,0]$, we have
$$\abs{\frac{e^{pt}}{p(\lambda_k+\vartheta(p))} \left\langle A_*u_{0,n},\phi_k\right\rangle_{L^2(\Omega;\rho dx)}}\leq C\max\left(|p|^{-\alpha_0+\epsilon-1},|p|^{-\alpha_0-1}\right)\abs{\left\langle A_*u_{0,n},\phi_k\right\rangle_{L^2(\Omega;\rho dx)}},$$
$$\abs{\partial_t\left(\frac{e^{pt}}{p(\lambda_k+\vartheta(p))} \left\langle A_*u_{0,n},\phi_k\right\rangle_{L^2(\Omega;\rho dx)}\right)}\leq C\max\left(|p|^{-\alpha_0+\epsilon},|p|^{-\alpha_0}\right)\abs{\left\langle A_*u_{0,n},\phi_k\right\rangle_{L^2(\Omega;\rho dx)}}.$$
Combining this with the fact that $u_{0,n}\in D(A_*)$, we deduce that
\bel{t2c}\partial_tu_n(t,\cdot)=-\sum_{k=1}^\infty \frac{1}{2i\pi}  \int_{\gamma(\delta,\theta)}   \frac{e^{pt}}{\lambda_k+\vartheta(p)} \left\langle A_*u_{0,n},\phi_k\right\rangle_{L^2(\Omega;\rho dx)}\phi_k=-S_{1,\mu}(t)Au_{0,n}.\ee
Applying Lemma \ref{l4} and the fact that $Au_{0,n}\in L^2(\Omega)$, we deduce that the map $u_n$ is lying in $W^{1,1}_{loc}(\R_+;L^2(\Omega))$. Moreover, we have 
$$u_n(0,\cdot)-u_{0,n}=-\sum_{k=1}^\infty \left(\frac{1}{2i\pi}  \int_{\gamma(\delta,\theta)}   \frac{1}{p(\lambda_k+\vartheta(p))}dp\right) \left\langle A_*u_{0,n},\phi_k\right\rangle_{L^2(\Omega;\rho dx)}\phi_k.$$
Applying the arguments used at the end of the proof of \cite[Proposition 2.1]{LKS}, we obtain
$$\frac{1}{2i\pi}  \int_{\gamma(\delta,\theta)}   \frac{1}{p(\lambda_k+\vartheta(p))}dp=0,\quad k\in\mathbb N$$
and it follows that    $I_K[u_n-u_{0,n}]\in W^{1,1}_{loc}(\R_+;L^2(\Omega))$ satisfies
\begin{equation}\label{t2d}
I_K[u_n-u_{0,n}](0,x)=0,\quad  x\in\Omega.
\end{equation}
This proves that $u_n$ satisfies condition (ii) of Definition \ref{d1}. Combining Lemma \ref{l4}, Proposition \ref{p1} with the arguments used in the last step of the proof of Theorem \ref{t2}, we can show that $u_n$ satisfies also condition (i). In the same way, using Lemma \ref{l4}, Proposition \ref{p1} and repeating the arguments used in the last step of the proof of Theorem \ref{t2}, we can show that, by density these properties can be extended to $u$. This proves  that the Laplace-weak solution $u$ of \eqref{eq1}, given by \eqref{di1}, fulfills the condition (i) and (ii) of Definition \ref{d1} and it completes the proof of Theorem \ref{t2}.\qed

\section{Multiterm fractional diffusion equations}

In this section, we prove the unique existence of a weak
solution to the problem \eqref{eq1} as well as the equivalence between Definition \ref{d1} and \ref{d2} of weak and Laplace-weak solutions of \eqref{eq1} for weight $K$ given by \eqref{Kmultiple} with $1<\alpha_1<\ldots<\alpha_N<1$ and $\rho_j\in L^\infty(\Omega)$, $j=1,\ldots,N$, satisfying \eqref{eq-rho} with $\rho=\rho_j$. In contrast to variable order and distributed order fractional diffusion equations, we have not find any result in the mathematical literature showing the unique existence of Laplace-weak solutions for multiterm fractional diffusion equations. For this purpose, we will consider first the proof of this result.

For all $p\in\mathbb C\setminus(-\infty,0]$, we can consider the following operator
$$\left(A+\sum_{k=1}^N\rho_k(x)p^{\alpha_k}\right)^{-1}\in \mathcal B(L^2(\Omega)).$$
For all $\theta\in(0,\pi)$ we denote by $\mathcal D_\theta$ the following set $ \mathcal D_{\theta}:=\{re^{i\beta}:\ r>0,\ \beta\in (-\theta,\theta)\}$.
Inspired by \cite[Proposition 2.1]{KSY}, we start with the following properties of the above operator.
\begin{lem}\label{l5} Let $\theta\in(0,\pi)$. Then there exists a constant $C>0$ depending only on $\mathcal A$, $\rho_1,\ldots,\rho_N$, $\alpha_1,\ldots,\alpha_N$, $\Omega$ and $\theta$ such that
\begin{equation}\label{l5a} \norm{\left(A+\sum_{k=1}^N\rho_k(x)z^{\alpha_k}\right)^{-1}}_{B(L^2(\Omega))}\leq C|z|^{-\alpha_N},\quad z\in \mathcal D_\theta.\end{equation}
\end{lem}
\begin{proof} Let us observe that, since the spectrum of $A$ is discrete and contained into $\R_+$, it is enough to prove \eqref{l5a} with $z\in \mathcal D_\theta$ satisfying $|z|>1$. For this purpose, from now on we fix $z=re^{i\beta}$ with $r\in[1,+\infty)$, $\beta\in(-\theta,\theta)$ and we will show \eqref{l5a}. In all this proof $c_0$ and $C_0$ denote the constants appearing in \eqref{eq-rho}. We divide the proof of this result into two steps.

\textbf{Step 1:} In this step we will prove that for all $\beta\in(-\theta,\theta)\setminus \{0\}$, we have
\begin{equation}\label{l5d}
\left\| \left(A+\sum_{k=1}^N\rho_k(x)z^{\alpha_k}\right)^{-1} \right\|_{\mathcal B(L^2(\Omega))} \leq c_0^{-1} \max\left(|\sin(\alpha_1\beta)|^{-1},|\sin(\alpha_N\beta)|^{-1}\right) r^{-\alpha_N}.\end{equation}
For this purpose, we assume that $\beta \in (0,\theta)$, the case of $\beta \in (-\theta,0)$ being treated in a similar fashion. Let $B_\beta$ be the multiplier in $L^2(\Omega)$, by the function 
$$ b_\beta(x) := \left( \sum_{k=1}^N\rho_k(x) r^{\alpha_k} \sin(\beta \alpha_k) \right)^{1 \slash 2},\ x \in \Omega, $$
in such a way that $i B_\beta^2$ is the skew-adjoint part of the operator $A+\sum_{k=1}^N\rho_k(x)r^{\alpha_k}e^{i\beta\alpha_k}$.
Applying \eqref{eq-rho}, we obtain 
$$ 0 < c_0^{1\slash2} \min\left(\sin(\alpha_1\beta)^{1 \slash 2},\sin(\alpha_N\beta)^{1 \slash 2}\right) r^{\alpha_N \slash 2}\leq b_\beta(x),\quad x\in\Omega,$$
$$b_\beta(x)\leq (NC_0)^{1 \slash 2} \max\left(\sin(\alpha_1\beta)^{1 \slash 2},\sin(\alpha_N\beta)^{1 \slash 2}\right)r^{\alpha_N \slash 2},\ x \in \Omega.$$
Hence, the self-adjoint operator $B_\beta$ is bounded and boundedly invertible in $L^2(\Omega)$, with
\begin{equation}\label{l5b}
\| B_\beta^{-1} \|_{\mathcal B(L^2(\Omega))} \leq c_0^{-1\slash2} \max\left(\sin(\alpha_1\beta)^{-1 \slash 2},\sin(\alpha_N\beta)^{-1 \slash 2}\right) r^{-\alpha_N \slash 2}.
\end{equation}
Moreover, for each $z=r e^{i \beta}$, it holds true that
\begin{equation}\label{l5c}
A + \sum_{k=1}^N\rho_k(x)z^{\alpha_k} = B_\beta \left( B_\beta^{-1} H_{z} B_\beta^{-1} + i \right) B_\beta,,
\end{equation}
where $H_z:= A + \sum_{k=1}^N\rho_k(x) r^{\alpha_k}\cos(\beta \alpha_k)$. It is clear that  the operator $H_z$ is self-adjoint in $L^2(\Omega)$ with domain $D(H_z)=D(A)$, by the Kato-Rellich theorem. Thus, $B_\beta^{-1} H_z B_\beta^{-1}$ is self-adjoint in $L^2(\Omega)$ as well, with domain $B_\beta D(A)$.
Therefore, the operator $B_\beta^{-1} H_z B_\beta^{-1}+i$ is invertible in $L^2(\Omega)$ and satisfies the estimate
$$ \| ( B_\beta^{-1} H_z B_\beta^{-1}+i )^{-1} \|_{\mathcal B(L^2(\Omega))} \leq 1. $$
It follows from this and \eqref{l1c} that $A_q+\sum_{k=1}^N\rho_k(x)z^{\alpha_k}$ is invertible in $L^2(\Omega)$, 
with
$$ \left(A+\sum_{k=1}^N\rho_k(x)z^{\alpha_k}\right)^{-1} = B_\beta^{-1} (U_\beta^{-1} H_z B_\beta^{-1} +i )^{-1} B_\beta^{-1}, $$
showing that $\left(A+\sum_{k=1}^N\rho_k(x)z^{\alpha_k}\right)^{-1}$ maps $L^2(\Omega)$ into $B_\beta^{-1} D(B_\beta^{-1} H_z B_\beta^{-1})=D(A)$.  As a consequence, we infer from \eqref{l5b} that
$$\begin{aligned}
\left\| \left(A+\sum_{k=1}^N\rho_k(x)z^{\alpha_k}\right)^{-1} \right\|_{\mathcal B(L^2(\Omega))} &\leq \| (B_\beta^{-1} H_z B_\beta^{-1}+i )^{-1} \|_{\mathcal B(L^2(\Omega))} \| B_\beta^{-1} \|_{\mathcal B(L^2(\Omega))}^2 \\
&\leq c_0^{-1} \max\left(|\sin(\alpha_1\beta)|^{-1},|\sin(\alpha_N\beta)|^{-1}\right) r^{-\alpha_N}.\end{aligned}$$
From this last estimate we deduce \eqref{l5d}.

\textbf{Step 2:} We fix $\theta_*\in(0,\min(\theta,\pi/2))$ such that 
\begin{equation}\label{l5e}\frac{|\sin(\alpha_N\theta_*)|}{\cos(\alpha_N\theta_*)}\leq \frac{c_0}{2C_0N}.\end{equation}
In this step, we will prove that for all $\beta\in(-\theta_*,\theta_*)$ we have
\begin{equation}\label{l5f}
\left\| \left(A+\sum_{k=1}^N\rho_k(x)z^{\alpha_k}\right)^{-1} \right\|_{\mathcal B(L^2(\Omega))} \leq 2c_0^{-1} \cos(\alpha_N\theta_*)^{-1} r^{-\alpha_N},\end{equation}
where we recall that since $\theta_*\in(0,\min(\theta,\pi/2))$ we have $\cos(\alpha_N\theta_*)>0$.
Using the fact that the operator $A$ is positive, for all $v\in D(A)$ and $\beta\in(-\theta_*,\theta_*)$, we get
$$\begin{aligned}\norm{H_zv}_{L^2(\Omega)}\norm{v}_{L^2(\Omega)}&\geq\left\langle H_zv,v\right\rangle_{L^2(\Omega)}\\
\ &\geq \left\langle Av,v\right\rangle_{L^2(\Omega)}+\left\langle \sum_{k=1}^N\rho_k \cos(\beta \alpha_k)v,v\right\rangle_{L^2(\Omega)}\\
\ &\geq \int_\Omega \left(\sum_{k=1}^N\rho_kr^{\alpha_k} \cos(\beta \alpha_k)\right)|v|^2dx\\
\ &\geq \int_\Omega \rho_Nr^{\alpha_N} \cos(\beta \alpha_N)|v|^2dx\\
\ &\geq c_0\cos(\theta_*\alpha_N)r^{\alpha_N}\norm{v}_{L^2(\Omega)}^2.\end{aligned}$$
 Choosing $v=H_z^{-1}u$ in the above inequality, we obtain
\begin{equation}\label{l5g}\norm{H_z^{-1}u}_{L^2(\Omega)}\leq c_0^{-1}\cos(\alpha_N\theta_*)^{-1}r^{-\alpha_N}\norm{u}_{L^2(\Omega)}\end{equation}
and we deduce that $\norm{H_z^{-1}}_{\mathcal B(L^2(\Omega))}\leq c_0^{-1}\cos(\alpha_N\theta_*)^{-1}r^{-\alpha_N}$.
Therefore, for all $\beta\in(-\theta_*,\theta_*)$, we get 
$$\begin{aligned}\norm{-iH_z^{-1}B_\beta^{2}}_{\mathcal B(L^2(\Omega))}&\leq \norm{H_z^{-1}}_{\mathcal B(L^2(\Omega))}\norm{B_\beta}_{\mathcal B(L^2(\Omega))}^{2}\\
\ &\leq \frac{C_0Nr^{\alpha_N}\sin(\alpha_N\beta)}{c_0\cos(\theta_*\alpha_N)r^{\alpha_N}}\leq \frac{C_0N\sin(\alpha_N\theta_*)}{c_0\cos(\theta_*\alpha_N)}\end{aligned}$$
and, combining this with \eqref{l5e}, we find
$$\norm{-iH_z^{-1}B_\beta^{2}}_{\mathcal B(L^2(\Omega))}\leq \frac{1}{2}.$$
Therefore, for all $\beta\in(-\theta_*,\theta_*)$, the operator $(Id+iH_z^{-1}B_\beta^{2})^{-1}$ is invertible and we have
$$\norm{(Id+iH_z^{-1}B_\beta^{2})^{-1}}_{\mathcal B(L^2(\Omega))}\leq \frac{1}{1-\norm{-iH_z^{-1}B_\beta^{2}}_{\mathcal B(L^2(\Omega))}}\leq 2.$$
It follows that
$$\left(A+\sum_{k=1}^N\rho_k(x)z^{\alpha_k}\right)^{-1}=(Id+iH_z^{-1}B_\beta^{2})^{-1}H_z^{-1},\quad \beta\in(-\theta_*,\theta_*)$$
and applying \eqref{l5g} we deduce \eqref{l5f}. Combining \eqref{l5d} and \eqref{l5f}, we deduce \eqref{l5a} by choosing $$C=\min(c_0^{-1} \sin(\alpha_N\theta)^{-1}, c_0^{-1} \sin(\alpha_1\theta_*)^{-1},2c_0^{-1} \cos(\alpha_N\theta_*)^{-1}).$$ This completes the proof of the lemma.\end{proof}

We fix $\theta\in(\frac{\pi}{2},\pi)$, $\delta\in\R_+$ and, applying Lemma \ref{l5}, we consider  the operator $\mathcal R_j(t)\in\mathcal B(L^2(\Omega))$, $j=0,1$ and $t\in\R_+$, given by 
\begin{equation}\label{R0}
\mathcal R_0(t)h=\frac{1}{2i\pi}\int_{\gamma(\delta,\theta)}e^{t p}\left(A+\sum_{k=1}^N\rho_kp^{\alpha_k}\right)^{-1}\left(\sum_{k=1}^N\rho_kp^{\alpha_k-1}\right)h d p,\quad h\in L^2(\Omega),\ t\in\R_+,
\end{equation}
\begin{equation}\label{R1}
\mathcal R_1(t)h=\frac{1}{2i\pi}\int_{\gamma(\delta,\theta)}e^{t p}\left(A+\sum_{k=1}^N\rho_kp^{\alpha_k}\right)^{-1}h d p,\quad h\in L^2(\Omega),\ t\in\R_+
\end{equation}
Note that here since the map  $z\longmapsto\left(A+\sum_{k=1}^N\rho_k(x)z^{\alpha_k}\right)^{-1}$ is holomorphic on $\mathbb C\setminus(-\infty,0]$ as a map taking values in $\mathcal B(L^2(\Omega))$, the definition of $\mathcal R_j$, $j=0,1$, will be independent of the choice of $\delta$ and $\theta\in(\frac{\pi}{2},\pi)$. Let us consider
\bel{mul} u(t,\cdot)=\mathcal R_0(t)u_0+\int_0^t\mathcal R_1(t-s)F(s,\cdot)ds.\ee
Combining the arguments used in Lemma \ref{l1} with estimate \eqref{l5a}, we can show the following properties of the maps $\R_+\ni t\mapsto\mathcal R_j(t)$, $j=0,1$.

\begin{lem}\label{l6}
Let $\theta\in\left(\frac{\pi}{2},\pi\right)$. The maps $t\longmapsto S_j(t)$, $j=0,1$, defined by
\eqref{R0}-\eqref{R1} are lying in $L^1_{loc}(\R_+;\mathcal B(L^2(\Omega))$ and there
exists a constant $C>0$ depending only on $\mathcal A,\rho,\theta,\Omega$
such that the estimates
\begin{equation}\label{l6a}
\|\mathcal R_0(t)\|_{\mathcal B(L^2(\Omega))}\leq
C\max\left(t^{\alpha_N-\alpha_1},1\right),\quad t>0,
\end{equation}
\begin{equation}\label{6b}
\|\mathcal R_1(t)\|_{\mathcal B(L^2(\Omega))}\leq
C\max\left(t^{\alpha_N-1},1\right),\quad t>0,
\end{equation}
hold true.
\end{lem}

This proves that, for $u_0\in L^2(\Omega)$ and $F\in\mathcal J$, $u$ given by \eqref{mul} is lying in $L^1_{loc}(\R_+;L^2(\Omega))$.
Let us prove that this function $u$ is the unique Laplace weak solution of \eqref{eq1} when $K$ is given by \eqref{Kmultiple}.
For this purpose, we need two intermediate results.

Combining the result of Lemma \ref{l5} with \cite[Theorem 1.1.]{KSY}, we deduce  the following.

\begin{lem}\label{l7} Let $u_0\in L^2(\Omega)$ and $F\in L^\infty(\R_+;L^2(\Omega)$. Then, the function
\bel{fin}v(\cdot,t)=\mathcal R_0(t)u_0+\int_0^t\mathcal R_1(t-s)F(s,\cdot)ds+ BF(t,\cdot),\quad t\in\R_+,\ee
is the unique   Laplace weak solution of \eqref{eq1}.
Here $B$ is defined by
$$Bh=\frac{1}{2i\pi}\int_{\gamma(\delta,\theta)}p^{-1}\left(A+\sum_{k=1}^N\rho_kp^{\alpha_k}\right)^{-1}h d p.$$

\end{lem}

\begin{proof}  From now on, for any Banach space $Y$  we denote by $\cS'(\R_+;Y)$ the set of temperate distributions supported in $[0,+\infty)$ taking values in $Y$. Since the proof of this result is rather long and similar to \cite[Theorem 1.1.]{KSY}, we only give the main idea of its proof when $u_0\equiv0$.

In the first step of this proof we introduce the following family of operators acting in 
$L^2(\Omega)$,
$$ 
\widetilde{W}(p):=p^{-1}\left(A+\sum_{k=1}^N\rho_kp^{\alpha_k}\right)^{-1} ,\ p \in \mathbb \C 
\setminus \R_-.
$$
Combining Lemma \ref{l1} with the arguments used in \cite[Lemma 2.3.]{KSY}, we can define the map 
\bel{bm} 
\mathcal R_2(t):=  \frac{1}{2i\pi} \int_{i\infty}^{i\infty} e^{t p} \widetilde{W}(p+1) dp 
= \frac{1}{2 \pi} \int_{-\infty}^{+\infty} e^{ i t\eta} \widetilde{W}(1 + i \eta ) d \eta
, \quad t\in\R\ee
and show that $\mathcal R_2\in L^\infty(\R; \mathcal B(L^2(\Omega))) \cap \cS'(\R_+;\cB(L^2(\Omega)))$. Moreover, combining Lemma \ref{l1} with Theorem 19.2 and the following remark in \cite{Ru}, we deduce that  
$\widehat{\mathcal R_2}(p)=\widetilde{W}(p+1)$ for all $p\in\mathbb C_+$. 
As a consequence, the operator
$$
\mathcal R_3(t)=e^{t} \mathcal R_2(t) =\frac{1}{2i\pi} \int_{-i\infty}^{i\infty} e^{t (p+1)} 
\widetilde{W}(p+1) dp= \frac{1}{2i\pi} \int_{1-i\infty}^{1+i\infty} e^{t p} \widetilde{W}(p) dp,
\ t \in \R
$$ 
verifies
$\widehat{\mathcal R_3}(p)=\widehat{\mathcal R_2}(p-1)=\widetilde{W}(p)$ for all $p\in\{z\in\mathbb C;\ \re z  \in (1,+\infty) \}$.
Following \cite[Lemma 2.4.]{KSY}, we can prove that

\bel{t6a}
\mathcal R_3(t)=\frac{1}{2 i \pi} \int_{\gamma(\delta,\theta)} e^{t p} \widetilde{W}(p) dp,\ t 
\in \R_+,
\ee
and $\mathcal R_3\in\cS'(\R_+;\cB(L^2(\Omega)))\cap L^1_{loc}(\R_+;\cB(L^2(\Omega)))$. Using the fact that $\widehat{\mathcal R_3}(p)=\widetilde{W}(p)$ for all $p\in\{z\in\mathbb C;\ \re z  \in (1,+\infty) \}$, we deduce that 
\bel{lap11}
\widehat{\mathcal R_3\psi}(p)= \widetilde{W}(p)\psi,\ p \in \C_+,\ \psi\in L^2(\Omega).
\ee

We denote by $\tilde{F}$  the extension of a function $F$ by $0$ on 
$( \Omega\times \R) \setminus (\Omega\times\R_+ )$. Consider the convolution in time of $S_2$ with $\tilde{F}$ given by
$$ 
(\mathcal R_3*\tilde{F})(x,t) =\int_0^t \mathcal R_3(t-s) F(s,x) \mathds{1}_{\R_+}(s) ds.
\ (t,x) \in \R\times\Omega  , 
$$
We show that   $\mathcal R_3*\tilde{F}\in \cS'(\R_+;L^2(\Omega))$ and 
$$
\widehat{\mathcal R_3*\tilde{F}}(p)=\widehat{\mathcal R_3}(p) \widehat{\tilde{F}}(p),\ 
p\in\mathbb C_+,
$$
with $\widehat{\mathcal R_3}(p)=\int_0^{+\infty} \mathcal R_3(t)e^{-pt}dt$ and $\widehat{\tilde{F}}(p)
= \int_0^{+\infty} F(t)e^{-pt}dt$.
Thus, setting $\tilde{v}:=\pd_t(\mathcal R_3*\tilde{F})\in \cS'(\R_+;L^2(\Omega))$, 
we derive from \eqref{lap11} that
$$
\widehat{\tilde{v}}(p) = p\widehat{\mathcal R_3*\tilde{F}}(p) = p\widehat{\mathcal R_3}(p)\widehat{\tilde{F}}(p)=\left(A+\sum_{k=1}^N\rho_kp^{\alpha_k}\right)^{-1} \widehat{\tilde{F}}(p),\ p \in \C_+.
$$
Therefore, the proof will be completed if we show that $\tilde{v}=v$ with $v$ given by \eqref{fin}. For this purpose, applying \eqref{l6a}, we deduce that
$$
(\mathcal R_3*\tilde{F})(t)=\frac{1}{2 i \pi} \int_{\gamma(\delta,\theta)} g(t,p) 
dp,\ t \in \R_+ 
$$
with
\bel{a15b} 
g(t,p):= \int_0^te^{(t-s)p} p^{-1}\left(A+\sum_{k=1}^N\rho_kp^{\alpha_k}\right)^{-1}
\tilde{F}(s,\cdot)ds,\  p \in \gamma(\delta,\theta).
\ee
Therefore, for a.e. $t \in \R_+$ and all $p \in  \gamma(\delta,\theta)$, 
we have 
\beas
\pd_t g(t,p)
& =& \int_0^te^{(t-s)p} \left(A+\sum_{k=1}^N\rho_kp^{\alpha_k}\right)^{-1}
\tilde{F}(s,\cdot)ds 
+ p^{-1}\left(A+\sum_{k=1}^N\rho_kp^{\alpha_k}\right)^{-1}
\tilde{F}(t,\cdot),
\eeas
and consequently
$$
\norm{\pd_t g(t,p)}_{L^2(\Omega)}
\leq \norm{\left(A+\sum_{k=1}^N\rho_kp^{\alpha_k}\right)^{-1}}_{\cB(L^2(\Omega))} 
\left(  \int_0^t e^{s\re p} ds +| p |^{-1} \right) \norm{\tilde{F}}
_{L^\infty(\R;L^2(\Omega))}.
$$
From this and \eqref{l5a}, it follows that
$$
\norm{\pd_t g(t,p)}_{L^2(\Omega)}
\leq  C| p |^{-(1+ \alpha_N)} \norm{\tilde{F}}
_{L^\infty(\R;L^2(\Omega))}=C| p |^{-(1+ \alpha_N)} \norm{F}
_{L^\infty(\R_+;L^2(\Omega))}. 
$$
As a consequence, the mapping $p \mapsto \pd_t g(t,p) \in 
L^1(\gamma(\delta,\theta);L^2(\Omega))$ for any fixed $t \in \R_+$ 
and we have $\tilde{v}(t)=\partial_t[\mathcal R_3*\tilde{F}](t)=\frac{1}{2 i \pi} 
\int_{\gamma(\delta,\theta)} \pd_t g(t,p) dp$, or equivalently
$$
\tilde{v}(\cdot,t)=\frac{1}{2 i \pi} \int_{\gamma(\delta,\theta)}\left(\int_0^t 
e^{(t-s)p} \left(A+\sum_{k=1}^N\rho_kp^{\alpha_k}\right)^{-1}F(s)ds+p^{-1}\left(A+\sum_{k=1}^N\rho_kp^{\alpha_k}\right)^{-1}F(t)\right)dp
$$
in virtue of \eqref{a15b}. Now, applying the Fubini theorem to the right-hand 
side of the above identity, we obtain that $\tilde{v}=v$ with $v$ given by \eqref{fin}. Combining this with Lemma \ref{l6}, we deduce that $v\in L^1_{loc}(\R_+;L^2(\Omega))$ and it is the unique Laplace-weak solution of \eqref{eq1} in the sense of Definition \ref{d2}.
\end{proof}

We can extend the result of Lemma \ref{l7}  as follows.

\begin{lem}\label{l8}  Let $u_0\in L^2(\Omega)$ and $F\in L^\infty(\R_+;L^2(\Omega))$. Then, the function $u$ given by \eqref{mul} is the unique   Laplace weak solution of \eqref{eq1}.
\end{lem}
\begin{proof} According to Lemma \ref{l6} and \ref{l7}, we only need to  show that here the map $B$ appearing in Lemma \ref{l6} will be equal to zero.  
To see this let us observe that, fixing $\delta<1$ and applying the Cauchy formula,  for any  $R>1$ and $h\in L^2(\Omega)$, we have
$$\frac{1}{2i\pi}\int_{\gamma(\delta,R,\theta)}p^{-1}\left(A+\sum_{k=1}^N\rho_kp^{\alpha_k}\right)^{-1}h d p=\frac{1}{2i\pi}\int_{\gamma_0(R,\theta)}p^{-1}\left(A+\sum_{k=1}^N\rho_kp^{\alpha_k}\right)^{-1}h d p$$
with
$$\gamma(\delta,R,\theta):=\gamma_-(\delta,R,\theta)\cup\gamma_0(\delta,\theta)\cup\gamma_+(\delta,R,\theta)$$
oriented in the counterclockwise direction, where
\[\gamma_\pm(\delta,R,\theta):=\{r e^{\pm i\theta}:\  r\in(\delta,R)\}.
\]
Sending $R\to+\infty$, we obtain 
$$Bh=\lim_{R\to+\infty}\frac{1}{2i\pi}\int_{\gamma_0(R,\theta)}p^{-1}\left(A+\sum_{k=1}^N\rho_kp^{\alpha_k}\right)^{-1}h d p.$$
On the other hand applying Lemma \ref{l1}, we deduce that
$$\begin{aligned}&\norm{\frac{1}{2i\pi}\int_{\gamma_0(R,\theta)}p^{-1}\left(A+\sum_{k=1}^N\rho_kp^{\alpha_k}\right)^{-1}h d p}_{L^2(\Omega)}\\
&\leq C\int_{-\theta}^{\theta}\norm{\left(A+\sum_{k=1}^N\rho_k(Re^{i\beta})^{\alpha_k}\right)^{-1}}_{\mathcal B(L^2(\Omega))}\norm{h}_{L^2(\Omega)} d \beta\\
\ &\leq CR^{-\alpha_N}\norm{h}_{L^2(\Omega)}.\end{aligned}$$
Therefore, we have
$$Bh=\lim_{R\to+\infty}\frac{1}{2i\pi}\int_{\gamma_0(R,\theta)}p^{-1}\left(A+\sum_{k=1}^N\rho_kp^{\alpha_k}\right)^{-1}h d p\equiv0.$$
\end{proof}

Combining Lemma \ref{l6}, \ref{l8} with the density arguments used in Proposition \ref{p1}, we obtain the following results about the unique existence of Laplace-weak solutions for \eqref{eq1}.

\begin{prop}\label{p3} Assume that the conditions \eqref{ell}-\eqref{eq-rho} are fulfilled.
Let $u_0\in L^2(\Omega)$, $F\in\mathcal J$,  $1<\alpha_1<\ldots<\alpha_N<1$, $\rho_j\in L^\infty(\Omega)$, $j=1,\ldots,N$, satisfy \eqref{eq-rho} with $\rho=\rho_j$, and let $K$ be given by \eqref{Kmultiple}. Then there exists a unique Laplace-weak
solution $u\in L^1_{loc}(\R_+;L^2(\Omega))$ to \eqref{eq1} given by \eqref{mul}.
\end{prop}

Combining Lemma \ref{l6}, Proposition \ref{p3} and mimicking the proof of Theorem \ref{t1}, we deduce Theorem \ref{t3}.

\section{Weak solution at finite time}

In a similar way to \cite{KSY,KY1,KY2,LKS},  following Definition \ref{d1} of weak solutions of \eqref{eq1}, we give the definition of weak solutions of the same problem at finite time. Namely, for $T>0$, let us consider the IBVP
\bel{eqq1}
\left\{ \begin{array}{rcll} 
(\rho(x) \partial^{K}_t+\cA ) v(t,x) & = & G(t,x), & (t,x)\in 
(0,T)\times\Omega,\\
 v(t,x) & = & 0, & (t,x) \in (0,T)\times\partial\Omega, \\  
 v(0,x) & = & u_0(x), & x \in \Omega,\ .
\end{array}
\right.
\ee
We give the following definition of weak solutions of \eqref{eqq1}.

\begin{defn}\label{d3} Let $F$ be the extension of the function $G$ by zero to $\R_+\times\Omega$. Then we call  weak solution of \eqref{eqq1}
the restriction on $(0,T)\times\Omega$ of the weak solution $u$ of the IBVP \eqref{eq1} in the sense of Definition \ref{d1}.
\end{defn}

Notice that, according to Definition \ref{d1}, any weak solutions $v$ of \eqref{eqq1} satisfies the following properties:\\
1) $v\in L^1(0,T;L^2(\Omega))$ and the identity
\begin{equation}\label{d3a}
\rho( x)D_t^K[v-u_0](t, x)+\mathcal A v(t, x)=G(t,x),\quad  x\in\Omega,\ t\in(0,T).
\end{equation}
holds true in the sense of distributions in $\R_+\times\Omega$.\\
2) We have $I_K [v-u_0]\in W^{1,1}(0,T;D'(\Omega))$ and the following initial condition
\begin{equation}\label{d3b}
 I_K[v-u_0](0, x)=0,\quad  x\in\Omega,
\end{equation}
is fulfilled. Moreover, applying the result of Theorem \ref{t1}, \ref{t2} and \ref{t3} we can show the unique existence of weak solution of \eqref{eqq1}. Let us also observe that the Definition \ref{d3} of weak solutions depends on the final time $T$. Nevertheless, we can show that  the unique weak solution of \eqref{eqq1} in the sense of Definition \ref{d3} is independent of  $T$ and by the same way of the extension of the source term $G$ under consideration in Definition \ref{d3}. All these properties can be sum up as follows.

\begin{thm}\label{t4} Assume that the condition of Theorem \ref{t1}, \ref{t2} and \ref{t3} be fulfilled and assume that the weight $K$ is given by \eqref{Kvariable} or \eqref{Kdistributed} or \eqref{Kmultiple}. Then for any $G\in L^1(0,T;L^2(\Omega))$ and $u_0\in L^2(\Omega)$, the IBVP \eqref{eqq1} admits a unique solution $v\in L^1(0,T;L^2(\Omega))$ in the sense of Definition \ref{d3}. Moreover, the unique weak solution of \eqref{eqq1} have a Duhamel type of representation given by: 
$$1)\quad v(t,\cdot)=S_{0}(t)u_0+ \int_0^t S_{1}(t-s) G(s,\cdot) ds,\quad t\in(0,T)$$
when $K$ is given by \eqref{Kvariable}. Here $S_0$ \emph{(}resp. $S_1$\emph{)} is defined by \eqref{S0} \emph{(}resp. \eqref{S1}\emph{)}.
$$2)\quad v(t,\cdot)=S_{0,\mu}(t)u_0+ \int_0^t S_{1,\mu}(t-s) G(s,\cdot) ds,\quad t\in(0,T)$$
when $K$ is given by \eqref{Kdistributed}. Here $S_{0,\mu}$ \emph{(}resp. $S_{1,\mu}$\emph{)} is defined by \eqref{Smu0}  \emph{(}resp. \eqref{Smu1}\emph{)}.
$$3)\quad v(t,\cdot)=\mathcal R_0(t)u_0+ \int_0^t \mathcal R_1(t-s) G(s,\cdot) ds,\quad t\in(0,T)$$
when $K$ is given by \eqref{Kmultiple}. Here $\mathcal R_0$ \emph{(}resp. $\mathcal R_1$\emph{)} is defined by \eqref{R0}  \emph{(}resp. \eqref{R1}\emph{)}.\\
Finally, the solution of  the IBVP \eqref{eqq1} in the sense of Definition \ref{d3} is independent of the choice of the final time $T$.
\end{thm}
\begin{proof} The proof the first two claims of this theorem are a direct consequence of Theorem \ref{t1}, \ref{t2} and \ref{t3} and the discussion in Section 2, 3, 4 for the representation of solutions. Therefore, we only need to prove that the unique solution  of  the IBVP \eqref{eqq1} in the sense of Definition \ref{d3} is independent of the choice of the final time $T$. For this purpose, let us consider $T_1<T_2$ and $G\in L^1(0,T_2;L^2(\Omega))$. For $j=1,2$, consider $v_j$ the weak solution of the IBVP \eqref{eqq1} with $T=T_j$ in the sense of Definition \ref{d3}. In order to prove that the solutions of \eqref{eqq1} in the sense of Definition \ref{d3} are independent of $T$, we need  to show that the restriction of $v_2$ to $(0,T_1)\times\Omega$ coincides with $v_1$.
In view of the first claims of the theorem, one of the following identities hold true:
$$1)\quad v_j(t,\cdot)=S_{0}(t)u_0+ \int_0^t S_{1}(t-s) G(s,\cdot) ds,\quad t\in(0,T_j),$$
$$2)\quad v_j(t,\cdot)=S_{0,\mu}(t)u_0+ \int_0^t S_{1,\mu}(t-s) G(s,\cdot) ds,\quad t\in(0,T_j),$$
$$3)\quad v_j(t,\cdot)=\mathcal R_0(t)u_0+ \int_0^t \mathcal R_1(t-s) G(s,\cdot) ds,\quad t\in(0,T_j).$$
Thus, in each case we deduce that
$$v_1(t,x)=v_2(t,x),\quad t\in(0,T_1),\ x\in\Omega.$$
This shows that the unique solution  of  the IBVP \eqref{eqq1} in the sense of Definition \ref{d3} is independent of the choice of the final time $T$.\end{proof}

\section*{Acknowledgments}

 This work was supported by  the French National Research Agency ANR (project MultiOnde) grant ANR-17-CE40-0029.
 
\end{document}